\documentclass[a4paper, 11pt]{article}
\usepackage{upgreek}
\usepackage[all]{xy}
\usepackage[utf8]{inputenc}
\usepackage{amsmath,bm}
\usepackage{amssymb}
\usepackage{amsthm}
\usepackage{mathtools}
\usepackage{cases}
\usepackage{float}
\usepackage{here}
\usepackage{comment}
\usepackage{ulem}
\usepackage[T1]{fontenc}
\usepackage{textcomp}
\usepackage{graphicx, color}
\usepackage{tikz,epic,eso-pic}

\theoremstyle{plain}
\newtheorem{thm}{Theorem}[section]
\newtheorem{lem}[thm]{Lemma}
\newtheorem{cor}[thm]{Corollary}
\newtheorem{prop}[thm]{Proposition}

\theoremstyle{definition}
\newtheorem{df}[thm]{Definition}

\newtheorem{eg}[thm]{Example}

\theoremstyle{remark}
\newtheorem{rem}[thm]{Remark}

\numberwithin{equation}{section}
\allowdisplaybreaks[3]

\newcommand{\R}{\mathbb{R}}

\newcommand{\Z}{\mathbb{Z}}

\newcommand{\Ob}{\mathrm{Ob}}
\newcommand{\Mor}{\mathrm{Mor}}
\newcommand{\Cat}{{\sf Cat}}

\newcommand{\Hc}{{\sf H}}
\newcommand{\e}{\varepsilon}

\newcommand{\Hom}{{\sf Hom}}
\newcommand{\Fib}{{\sf Fib}}
\newcommand{\Tor}{{\sf Tors}}
\newcommand{\PMet}{{\sf PMet}}
\newcommand{\EPMet}{{\sf EPMet}}
\newcommand{\EQMet}{{\sf E\uppsi}{\sf Met}}
\newcommand{\Aut}{{\sf Aut}}

\newcommand{\Grph}{{\sf Grph}}
\newcommand{\wGrph}{{\sf wGrph}}

\newcommand{\MGrp}{{\sf MGrp}}
\newcommand{\EMGrp}{{\sf EMGrp}}
\newcommand{\NGrp}{{\sf NGrp}}

\newcommand{\G}{\mathcal{G}}
\renewcommand{\H}{\mathcal{H}}
\newcommand{\hF}{\widehat{F}}

\newcommand{\core }{{\sf core}}
\newcommand{\Met}{{\sf Met}}
\newcommand{\EMet}{{\sf EMet}}
\newcommand{\QMet}{{\sf \uppsi Met}}

\newcommand{\wt}{\widetilde}

\newcommand{\wh}{\widehat}
\newcommand{\too}{\longrightarrow}

\usepackage{xcolor}

\title{
Classification of metric fibrations
}
\author{Yasuhiko \textsc{Asao}\thanks{Department of Applied Mathematics, Fukuoka University \texttt{asao@fukuoka-u.ac.jp}}}
\date{\today}
\begin{document}
\maketitle
\begin{abstract}
 In this paper, we study a notion of `fibration for metric spaces', called {\it metric fibration}, that was originally introduced by Leinster \cite{L3} in the study of {\it magnitude}. He showed that the magnitude of a metric fibration splits into the product of those of the fiber and the base, which is analogous to the case for Euler characteristic and topological fiber bundles. His idea and our approach is based on Lawvere's suggestion of viewing a metric space as an enriched category \cite{La}. Actually, the metric fibrations are the restriction of the enriched {\it Grothendieck fibrations} \cite{Gr} to metric spaces \cite{A0}.  We give a complete classification of metric fibrations by several means, which are parallel to those used for topological fiber bundles. That is, the classification of metric fibrations is reduced to that of `principal fibrations', which is done by the `1-\v{C}ech cohomology' in an appropriate sense. Here we introduce the notion of {\it torsors} in the category of metric spaces, and the discussions are analogous to those in sheaf theory. Further, we can define the `fundamental group' $\pi^m_1(X)$ of a metric space $X$, which is a group-like object in metric spaces, such that the conjugation classes of homomorphisms $\Hom(\pi^m_1(X), \G)$ corresponds to the isomorphism classes of `principal $\G$-fibrations' over $X$. In other words, the latter are classified like topological covering spaces.
 

\end{abstract}

\section{Introduction}

The idea of {\it metric fibration} was first introduced by Leinster in the study of magnitude \cite{L3}. The magnitude theory that he coined can be considered as a promotion of Lawvere's suggestion of viewing a metric space as a $[0, \infty]$-enriched category. The magnitude of a metric space was defined as a special case of the `Euler characteristic of enriched categories'. In fact, he showed that the magnitude of a metric fibration splits into the product of those of the fiber and the base (Theorem 2.3.11 of \cite{L3}), which is analogous to the case of topological fiber bundles.   Later, the author \cite{A0} pointed out that metric fibration can actually be seen as enriched {\it Grothendieck fibrations}, see  \cite{Gr}, when restricted to metric spaces. Here we deal with small categories and metric spaces from a unified viewpoint, namely as {\it filtered set enriched categories}. By this approach, we can expect to obtain novel ideas for the study of metric spaces by transferring well understood concepts in category theory, and vice versa. 

As an example, the following Figure $1$ is one of the simplest non-trivial metric fibrations. Note that we consider connected graphs as metric spaces by taking the shortest path metric (see also Proposition \ref{leftadjoint}). Both graphs are metric fibrations over the complete graph $K_3$ with fiber $K_2$ as shown in Example 5.29 of \cite{A0}. Further, they have the same magnitude as pointed out in Example 3.7 of \cite{L1}. In Proposition 5.30 of \cite{A0}, it is shown that the right one is the only non-trivial metric fibration over $K_3$ with fiber $K_2$. Here, `trivial' means that it is the cartesian product of graphs. On the other hand, any metric fibration over a four cycle graph $C_4$, or more generally an even cycle graph, is shown to be trivial in the same proposition. 
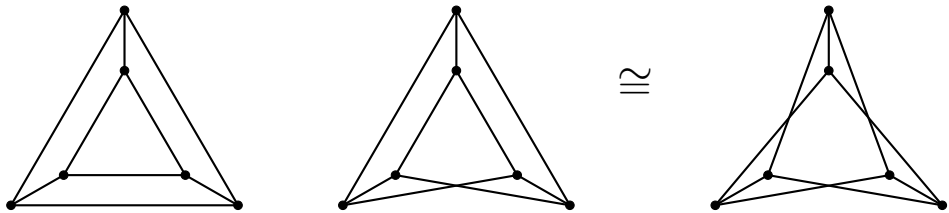
\begin{figure}[htbp]
\centering
\begin{tikzpicture}[scale=0.8]
\filldraw[fill=black, draw=black] (1, 1.732) circle (2pt) node[left] {} ;
\filldraw[fill=black, draw=black] (2, 0) circle (2pt) node[left] {} ;
\filldraw[fill=black, draw=black] (0, 0) circle (2pt) node[left] {} ;
\filldraw[fill=black, draw=black] (1, 2.732) circle (2pt) node[right] {} ;
\filldraw[fill=black, draw=black] (-0.866, -0.5) circle (2pt) node[right] {} ;
\filldraw[fill=black, draw=black] (2.866, -0.5) circle (2pt) node[left] {} ;

\draw[thick] (1, 1.732)--(2, 0);
\draw[thick] (1, 1.732)--(0, 0);
\draw[thick] (0, 0)--(2, 0);
\draw[thick] (0, 0)--(-0.866, -0.5);
\draw[thick] (1, 1.732)--(1, 2.732);
\draw[thick] (2, 0)--(2.866, -0.5);
\draw[thick] (2.866, -0.5)--(-0.866, -0.5);
\draw[thick] (1, 2.732)--(-0.866, -0.5);
\draw[thick] (1, 2.732)--(2.866, -0.5);

\end{tikzpicture}
\hspace{1cm}
\centering
\begin{tikzpicture}[scale=0.8]
\filldraw[fill=black, draw=black] (1, 1.732) circle (2pt) node[left] {} ;
\filldraw[fill=black, draw=black] (2, 0) circle (2pt) node[left] {} ;
\filldraw[fill=black, draw=black] (0, 0) circle (2pt) node[left] {} ;
\filldraw[fill=black, draw=black] (1, 2.732) circle (2pt) node[right] {} ;
\filldraw[fill=black, draw=black] (-0.866, -0.5) circle (2pt) node[right] {} ;
\filldraw[fill=black, draw=black] (2.866, -0.5) circle (2pt) node[left] {} ;
\filldraw[fill=white, draw=white] (4.4, 1.5) circle (2pt) node[left] {{\LARGE $\cong$}} ;

\draw[thick] (1, 1.732)--(2, 0);
\draw[thick] (1, 1.732)--(0, 0);
\draw[thick] (2, 0)--(-0.866, -0.5);
\draw[thick] (0, 0)--(-0.866, -0.5);
\draw[thick] (1, 1.732)--(1, 2.732);
\draw[thick] (2, 0)--(2.866, -0.5);
\draw[thick] (2.866, -0.5)--(0, 0);
\draw[thick] (1, 2.732)--(-0.866, -0.5);
\draw[thick] (1, 2.732)--(2.866, -0.5);

\end{tikzpicture}
\hspace{0.3cm}
\centering
\begin{tikzpicture}[scale=0.8]
\filldraw[fill=black, draw=black] (1, 1.732) circle (2pt) node[left] {} ;
\filldraw[fill=black, draw=black] (2, 0) circle (2pt) node[left] {} ;
\filldraw[fill=black, draw=black] (0, 0) circle (2pt) node[left] {} ;
\filldraw[fill=black, draw=black] (1, 2.732) circle (2pt) node[right] {} ;
\filldraw[fill=black, draw=black] (-0.866, -0.5) circle (2pt) node[right] {} ;
\filldraw[fill=black, draw=black] (2.866, -0.5) circle (2pt) node[left] {} ;

\draw[thick] (1, 1.732)--(2.866, -0.5);
\draw[thick] (1, 1.732)--(-0.866, -0.5);
\draw[thick] (2, 0)--(-0.866, -0.5);
\draw[thick] (0, 0)--(-0.866, -0.5);
\draw[thick] (1, 1.732)--(1, 2.732);
\draw[thick] (2, 0)--(2.866, -0.5);
\draw[thick] (2.866, -0.5)--(0, 0);
\draw[thick] (1, 2.732)--(0, 0);
\draw[thick] (1, 2.732)--(2, 0);

\end{tikzpicture}
\label{graphbdl}
\caption{The left is $K_3\times K_2$, and the right is isomorphic to $K_{3, 3}$. They both have magnitude equal to  $\frac{6}{1 + 3q + 2q^2}$.}
\end{figure}

In this paper, we give a complete classification of metric fibrations by several means, which are parallel to those used to classify topological fiber bundles. Namely, we define `principal fibrations', `fundamental groups' and `a $1$-\v{C}ech cohomology' for metric spaces, and obtain an equivalence between categories of these objects. Roughly speaking, we obtain an analogy of the following correspondence in the case of topological fiber bundles with a discrete structure group.

\begin{equation*}
\xymatrix{
\text{Fiber bundles over $X$ with structure group $G$} \ar@{<->}[d] \\
\text{Principal $G$-bundles over $X$ ($G$-torsors)} \ar@{<->}[d] \\
\ 
}
\end{equation*}
\vspace{-0.55cm}
\begin{equation*}
\xymatrix{
[X, BG] \cong \Hom(\pi_1(X), G)/{\sf conjugation} \ar@{<->}[d] \\
{\sf H}^1(X, G)
}
\end{equation*}

We give more details below. First recall that any Grothendieck fibration (in the usual sense) over a small category $C$ can be obtained from a lax functor $C \too \Cat$, by a procedure known as the {\it Grothendieck construction} \cite{Gr2}. In \cite{A0}, it is shown that any metric fibration over a metric space $X$ can be obtained from a `lax functor' $X \too \Met$ that is called {\it metric action} (Definition \ref{metacdef}). Here $\Met$ is the category of metric spaces and Lipschitz maps. Whereas metric fibrations can be defined by ‘a lifting property’ in a model categorical spirit, metric actions are best understood by ‘transformation functions’. The following shows that these two viewpoints are in fact equivalent.

\begin{thm}[Proposition \ref{metfib}]
For a metric space $X$, the Grothendieck construction gives an equivalence  of categories
\[
\Met_X \simeq \Fib_X,
\]
where we denote the category of metric actions $X \too \Met$ by $\Met_X$ and the category of metric fibrations over $X$ by $\Fib_X$ (Definitions \ref{metacdef}, \ref{metfibdef}). 
\end{thm}

We can define a category $\Tor_X^\G$ that consists of `principal $\G$-fibrations' (Definition \ref{tordef}). We call it the {\it category of} $\G$-{\it torsors}. We can also define a subcategory $\PMet_X^\G$ of $\Met_X$, that is the counterpart of $\Tor_X^\G$ (Definition \ref{pmetdef}). The objects of the category $\PMet_X^\G$ consist of a metric action $X \too \Met$ taking a group $\G$, not just a metric space, as the value. Then we have the following.

\begin{thm}[Proposition \ref{pmettor}]
For a ``metric group'' $\G$, the Grothendieck construction gives an  equivalence of categories 
\[
\PMet_X^\G \simeq \Tor_X^\G.
\]

\end{thm}

Here, the group $\G$ is not just a group but is a group-like object in $\Met$, which we call a {\it metric group} (Definition \ref{metgrpdef}). As an example of a metric group, we construct the {\it fundamental group $\pi_1^m(X)$ of a metric space} $X$ (Definition \ref{pi1def}). We also define a category $\Hom(\pi_1^m(X), \G)$ of homomorphisms $\pi_1^m(X) \too  \G$, where a morphism between homomorphisms is defined as a conjugation relation (Definition \ref{homcat}). Then we have the following.

\begin{thm}[Proposition \ref{hompri}]
We have an  equivalence of categories 
\[
\Hom (\pi^m_1(X, x_0), \G) \simeq \PMet^{\G}_X. 
\]

\end{thm}

As a corollary, we reprove Proposition 5.30 of \cite{A0} in the following form. We note that the notion of a metric group is equivalent to that of a {\it `normed group'} (Proposition \ref{eroff}). For a metric group $\G$, we denote the corresponding norm of an element $g \in \G$ by $|g| \in \Z_{\geq 0}$. 

\begin{prop}[Proposition \ref{cyclegraph}]
Let $C_n$ be the undirected $n$-cycle graph. Then we have
\[
\pi^m_1(C_n) \cong \begin{cases}\Z \text{ with } |1| = 1 & n : \text{odd}, \\ 0 & n : \text{even}. \end{cases}
\]
Hence we have that $\PMet_{C_n}^{\G} \simeq \begin{cases}\Hom (\Z, \G)  & n : \text{odd}, \\ 0 & n : \text{even}, \end{cases}$ for all metric group $\G$, which implies that there is only a trivial metric fibration over $C_{2n}$ and that there is at most one non-trivial metric fibration over $C_{2n+1}$.
\end{prop}

 Now, similarly to the topological case, we can define an `associated bundle construction' from a torsor and a metric space $Y$ (Corollary \ref{changefiber}). This construction gives the following.

\begin{thm}[Corollary \ref{pmetcore}]
Suppose that $Y$ is a bounded metric space. Then we have an  equivalence of categories 
\[
\PMet_X^{\Aut Y} \simeq {\sf core}\Fib_X^Y,
\]
where $\Fib_X^Y$ is the full subcategory of $\Fib_X$ that consists of metric fibrations with fiber $Y$ (Definition \ref{fiby}), and we denote the {\it core} of a category by ${\sf core}$ (Definition \ref{catconv} (4)).
\end{thm}
Here, we equip the group $\Aut Y$ of isometries of $Y$ with a metric group structure by $d_{\Aut Y}(f, g) = \sup_{y \in Y}d_Y(fy, gy)$ (Example \ref{auty}). However, for this we should suppose that $Y$ is a bounded metric space so that $d_{\Aut Y}$ is indeed a distance function. For the case of general metric fibrations, we must extend our arguments to {\it extended metric groups}, allowing $\infty$ as the value of the distance function (Definition \ref{extmet}). For these we obtain an essentially similar but extended result (Proposition \ref{extpmetcore}). 

Finally, we define a `$1$-\v{C}ech cohomology category' $\Hc^1(X, \G)$ of a $\G$-torsor $X$ (Definition \ref{cech}). This is analogous to the \v{C}ech cohomology constructed from the local sections of a principal bundle. Similarly to the topological case, we can construct a cocycle from a family of local sections (Proposition \ref{alpha}), and conversely we can construct a $\G$-torsor by pasting copies of $\G$ along a cocycle (Proposition \ref{pasting}). From this correspondence we have the following:

\begin{thm}[Corollary \ref{betaeq}]
We have an  equivalence of categories 
\[
\Hc^1(X; \G) \simeq  \Tor^{\G}_X.
\]

\end{thm}

\subsubsection*{Acknowledgements}
The author is grateful to Luigi Caputi and Emily Roff for fruitful and helpful comments and feedback on the first draft of the paper. In particular, Emily pointed out that the earlier version of the definition of metric groups is incorrect. She also checked author's poor English and gave many very helpful advice. The author also would like to thank Masahiko Yoshinaga and Shun Wakatsuki for valuable discussions and comments. Finally, he expresses his deep appreciation to anonymous referees and the editor for their careful reading and important suggestions.
\section{Conventions}
In this section, we review terms for categories, graphs, weighted graphs and metric spaces that are well-known but may not be commonly used. Those who are not familiar with the language of categories and functors may refer to \cite{Mc}. For categories and metric spaces, \cite{LS} may be a good reference.
\subsection{Categories}
In this article, we suppose that categories are locally small. We denote the object class of a category $C$ by $\Ob C$, and the set of all morphisms from $a$ to $b$ by $C(a, b)$ for each pair of objects $a, b \in \Ob C$. We denote the class of all morphisms in $C$ by $\Mor C$.
\begin{df}\label{catconv}
Let $C$ and $D$ be categories, and  $F : C \too D$ be a functor.
\begin{enumerate}
\item We say that $F$ is {\it faithful} if the map $F : C(a, b) \too D(Fa, Fb)$ is injective for all objects $a, b \in \Ob C$. We say that $F$ is {\it full} if the map $F : C(a, b) \too D(Fa, Fb)$ is surjective for all objects $a, b \in \Ob C$. We say that $F$ is {\it fully faithful} if it is faithful and full.
\item We say that $F$ is {\it split essentially surjective} if there is a family of isomorphisms $\{Fc \cong d \mid c \in \Ob C\}_{d \in \Ob D}$.
\item We say that $F$ is a {\it category equivalence} if there exists a functor $G : D \too C$ and natural isomorphisms $GF \cong {\rm id}_C$ and $FG \cong {\rm id}_D$. When there exists an  equivalence of categories $C \too D$, we say that {\it $C$ and $D$ are equivalent}.
\item We define a groupoid $\core C$ by $\Ob\  \core C = \Ob C$ and $\core C(a, b) = \{f \in C(a, b) \mid \text{ $f$ is an isomorphism}\}$ for all $a, b \in \Ob C$.
\end{enumerate}
\end{df}

The following are standard.
\begin{lem}
If a functor $F : C \too D$ is fully faithful and split essentially surjective, then it is an  equivalence of categories. \qed 
\end{lem}

\begin{lem}
An  equivalence of categories $F : C \too D $ induces an  equivalence of categories $\core F : \core C \too \core D$. \qed
\end{lem}

\begin{rem}
For a classification of objects of a category, we often want to consider `isomorphism classes of objects' and compare it with another category. However, in general, we cannot do that since the class of objects is not necessarily a set. Instead, we consider an  equivalence of categories $\core C \too \core D$ that  implies a bijection between isomorphism classes of objects if they are small. 
\end{rem}

\subsection{Metric spaces}
\begin{df}
\begin{enumerate}
\item A {\it pseudometric space} $(X, d)$ is a set $X$ equipped with a function $d : X \too \R_{\geq 0}$ satisfying that, for all $x, x', x'' \in X$, we have 
\begin{itemize}
\item $d(x, x) = 0$,
\item $d(x, x') = d(x', x)$,
\item $d(x, x') + d(x', x'') \geq d(x, x'')$.
\end{itemize}
\item A {\it Lipschitz map} $f : X \too Y$ between pseudometric spaces $X$ and $Y$ is a map satisfying that $d_Y(fx, fx') \leq d_X(x, x')$ for all $x, x' \in X$. We denote the category of pseudometric spaces and Lipschitz maps by $\QMet$. We call an isomorphism in $\QMet$ an {\it isometry}.
\item A {\it metric space} $(X, d)$ is a pseudometric space satisfying that 
\begin{itemize}
\item $d(x, x') = 0$ if and only if $x = x'$.
\end{itemize}
We denote the full subcategory of $\QMet$ that consists of metric spaces by $\Met$.
\end{enumerate}
\end{df}
\begin{df}
\begin{enumerate}
\item A {\it graph} $G$ is a pair of sets $(V(G), E(G))$ such that $E(G) \subset \{e \in 2^{V(G)}\mid \# e = 2\}$, where we denote the cardinality of a set by $\#$. We call an element of $V(G)$ a {\it vertex}, and an element of $E(G)$ an {\it edge}. A {\it graph homomorphism} $f : G \too H$ between graphs $G$ and $H$ is a map $f : V(G) \too V(H)$ such that $fe \in E(H)$ or $\# fe = 1$ for all $e \in E(G)$. We denote the category of graphs and graph homomorphisms by $\Grph$.
\item A {\it path} in a graph $G$ is a tuple $(x_0, \dots, x_n) \in V(G)^{n+1}$ for some $n\geq 0$ such that $\{x_i, x_{i+1}\} \in E(G)$ for all $0\leq i \leq n-1$. A {\it connected graph} $G$ is a graph such that for every $x,x' \in V(G)$ there exists a path $(x_0, \dots, x_n)$ with $x_0 = x$ and $x_n = x'$. We denote the full subcategory of $\Grph$ that consists of connected graphs by $\Grph_{\sf conn}$.
\item A {\it weighted graph} $(G, w_G)$ is a graph $G$ equipped with a function $w_G : E(G) \too \R_{\geq 0}$. A {\it weighted graph homomorphism} $f : G \too H$ between weighted graphs $G$ and $H$ is a graph homomorphism such that $w_H(fe) \leq w_G(e)$ for all $e \in E(G)$, where we stipulate that $w_H(fe) = 0$ if $\# fe = 1$. We denote the category of weighted graphs and weighted graph homomorphisms by $\wGrph$. We also denote the full subcategory of $\wGrph$ that consists of weighted graphs $(G, w_G)$ such that the graph $G$ is connected by $\wGrph_{\sf conn}$.
\end{enumerate}
\end{df}
\begin{df}
We define functors $\Met \too \QMet$ and $\wGrph_{\sf conn} \too \Grph_{\sf conn}$ by forgetting additional properties. We also define the functor $\QMet \too \wGrph_{\sf conn}$ that sends a pseudometric space $(X, d)$ to a weighted graph $(X, w_X)$ defined by $V(X) = X, E(X) = \{e \in 2^{X}\mid \# e = 2\}$ and $w_X \{x, x'\} = d(x, x')$. 
\end{df}
\begin{prop}\label{leftadjoint}
The above functors have left adjoints.
\end{prop}
\begin{proof}
We describe each functor $F$ in the following, and they are the left adjoint functors of each functor $G$ of the above since the unit and the counit give that $FGF = F$ and $GFG = G$.
\begin{itemize}
\item We define a functor $\Grph_{\sf conn} \too \wGrph_{\sf conn}$ by sending a connected graph to a weighted graph with $w = 0$. 
\item We define a functor $\wGrph_{\sf conn} \too \QMet$ by sending a weighted graph $(G, w_G)$ to a pseudometric space $(V(G), d_G)$ defined by 
\[
d_G(x, x') = \inf \cup_{n\geq 0}\{\sum_{i=0}^{n-1}w_G\{x_i, x_{i+1}\} \mid (x = x_0, \dots, x_n = x') \text{ is a path on } G\}.
\]

\item We define a functor $\QMet \too \Met$ by sending a pseudometric space $(X, d)$ to a metric space $({\sf KQ}X, \wt{d})$ defined as follows. We define an equivalence relation $\sim$ on $X$ by $x \sim x'$ if and only if $d(x, x') = 0$. We also define a function ${\sf KQ}X := X/\sim \too \R_{\geq 0}$ by $\wt{d}([x], [x']) = d(x, x')$. \qedhere
 \end{itemize}
\end{proof}
\begin{df}
For a pseudometric space $X$, we call the metric space ${\sf KQ}X$ the {\it Kolmogorov quotient} of $X$.
\end{df}
\begin{df}
\begin{enumerate}
\item For pseudometric spaces $(X, d_X)$ and $(Y, d_Y)$, we define a metric space called the {\it $L^1$-product} $(X\times Y, d_{X\times Y})$ by $d_{X\times Y}((x, y), (x', y')) = d_X(x, x') + d_Y(y, y')$ for all $x, x' \in X$ and $y, y' \in Y$.
\item For graphs $G$ and $H$, we define a graph called the {\it cartesian product} $G\times H$ by $V(G\times H) = V(G)\times V(H)$, and $\{(x, y), (x', y')\} \in E(G\times H)$ if and only if one of the following holds :
\begin{itemize}
\item $x = x'$ and $\{y, y'\} \in E(H)$,
\item $\{x, x'\} \in E(G)$ and $y = y'$,
\end{itemize}
for all $x, x' \in V(G)$ and $y, y' \in V(H)$.
\item For weighted graphs $(G, w_G)$ and $(H, w_H)$, we define a weighted graph $(G\times H, w_{G\times H})$ by $w_{G\times H}\{(x, y), (x', y')\} = w_G\{x, x'\} + w_H\{y, y'\}$ for all $\{(x, y), (x', y')\} \in E(G\times H)$, where $G\times H$ is the cartesian product of graphs and we stipulate that $w_G\{x, x\} = w_H\{y, y\} = 0$.
\end{enumerate}
\end{df}
These products make each category a symmetric monoidal category. 
\begin{prop}\label{monoidal}
The functors $\Met \too \QMet \too \wGrph_{\sf conn} \too \Grph_{\sf conn}$ and their left adjoints are strong monoidal except for the functor $ \QMet \too \wGrph_{\sf conn}$ that is lax monoidal.
\end{prop}
\begin{proof}
For the functors $\Met \too \QMet$ and $\wGrph_{\sf conn} \too \Grph_{\sf conn}$, the claim is obvious since they are inclusions. The claim is also obvious for the functor $ \Grph_{\sf conn} \too \wGrph_{\sf conn}$ by the definition. For the functor $\QMet \too \Met$, we define a map ${\sf KQ}(X\times Y) \too {\sf KQ}X\times {\sf KQ}Y$ by $[(x, y)] \mapsto ([x], [y])$. This is obviously natural and is an isometry since we have that $[(x, y)]\sim [(x', y')]$ if and only if $[x]\sim [x']$ and $[y]\sim [y']$. For the functor $F : \wGrph_{\sf conn} \too \QMet$, the identity on the set $F(G\times H) = F(G)\times F(H)$ is an isometry since 
\begin{align*}
&\ d_{w_{G\times H}}((x, y), (x', y'))  \\
&= \inf \cup_{n\geq 0}\{\sum_{i=0}^{n-1}w_{G\times H}\{(x_i, y_i), (x_{i+1}, y_{i+1})\} \mid  \\
&\ \  ((x, y) = (x_0, y_0), \dots, (x_n, y_n) = (x', y')) \text{ is a path on } G\times H\} \\
&= \inf \cup_{n\geq 0}\{\sum_{i=0}^{n-1}w_{G}\{x_i, x_{i+1}\} + w_{H}\{y_i, y_{i+1}\} \mid  \\
& \ \  ((x, y) = (x_0, y_0), \dots, (x_n, y_n) = (x', y')) \text{ is a path on } G\times H\} \\
&= \inf \cup_{n\geq 0}\{\sum_{i=0}^{n-1}w_{G}\{x_i, x_{i+1}\} \mid (x = x_0, \dots, x_n = x')\} \\
&+ \inf \cup_{m\geq 0}\{\sum_{i=0}^{m-1} w_{H}\{y_i, y_{i+1}\}  \mid (y = y_0, \dots, y_m = y')\} \\
&= d_{w_G}(x, x') + d_{w_H}(y, y') \\
&= d_{F(G)\times F(H)}((x, y), (x', y')),
\end{align*}
for all $x, x' \in V(G)$ and $y, y' \in V(H)$. It is obviously natural. Finally, for the functor $G : \QMet \too \wGrph_{\sf conn}$, the identity on the set $G(X)\times G(Y) = G(X\times Y)$ is a weighted graph homomorphism since it is an inclusion of graphs and preserves weightings. It is obviously natural. 
\end{proof}

\begin{df}
\begin{enumerate}
\item An {\it extended pseudometric space} is a set $X$ equipped with a function $d : X \too [0, \infty]$ that satisfies the same conditions for pseudometric spaces. In other words, it is a pseudometric space admitting $\infty$ as a value of distance. A {\it Lipschitz map} between extended pseudometric spaces is a distance non-increasing map. We denote the category of extended pseudometric spaces and Lipschitz maps by $\EQMet$.  We similarly define {\it extended metric spaces} and we denote the full subcategory of $\EQMet$ that consists of them by $\EMet$.
\item We define the $L^1$ product of extended pseudometric spaces similarly to that of pseudometric spaces. It makes the category $\EQMet$ a symmetric monoidal category.
\item We define functors $\EMet \too \EQMet$ and $\wGrph \too \Grph$ by forgetting additional properties. We also define the functor $\EQMet \too \wGrph$ similarly to the functor $\QMet \too \wGrph_{\sf conn}$ except that $\{x, x'\}$ does not span an edge for $x, x' \in X$ with $d(x, x') = \infty$.
\end{enumerate}
\end{df}
The following is immediate.
\begin{prop}
\begin{enumerate}
\item The functors $\EMet\too\EQMet\too \wGrph \too \Grph$ have left adjonts. Further, the following diagram is commutative, where the vertical functors are all inclusions.

\begin{equation*}
    \xymatrix{
    \EMet \ar[r] & \EQMet \ar[r]  & \wGrph \ar[r]& \Grph \\
    \Met \ar[u] \ar[r]& \QMet \ar[u]\ar[r]& \wGrph_{\sf conn} \ar[u]\ar[r]& \Grph_{\sf conn}\ar[u].
    }
\end{equation*}
\item The functors $\EMet \too \EQMet$  and $\wGrph \too \Grph$ are strong monoidal and the functor $\EQMet\too \wGrph$ is lax monoidal. \qed
\end{enumerate}
\end{prop}

\section{$\Met_X \simeq \Fib_X$}
In this section, we introduce two notions, the {\it metric action} and the {\it metric fibration}, and show the equivalence between them. The notion of metric fibation is originally introduced by Leinster \cite{L3} in the study of magnitude.  The metric action was introduced by the present author in \cite{A0}, and is the counterpart of {\it lax functors} in category theory, while the metric fibration is a generalization of the {\it Grothendieck fibration}. As written in the introduction, we can think of the Grothendieck (or metric) fibration as the definition of fibrations by `the lifting property', while the lax functor is the one by `the transformation functions'. 
\begin{df}\label{metacdef}
Let $X$ be a metric space. 
\begin{enumerate}
\item A {\it metric action} $F : X \too \Met$ consists of metric spaces $Fx \in \Met$ for all $x \in X$ and isometries $F_{xx'} : Fx \too Fx'$ for all $x, x' \in X$ satisfying the following for all $x, x', x'' \in X$ :
\begin{itemize}
\item $F_{xx} = {\rm id}_{Fx}$ and $F_{x'x} = F_{xx'}^{-1}$,
\item $d_{Fx''}(F_{x'x''}F_{xx'}a, F_{xx''}a) \leq d_X(x, x') + d_X(x', x'') - d_X(x, x'')$ for every $a \in Fx$.
\end{itemize}
\item A {\it metric transformation} $\theta : F \Longrightarrow G$ consists of Lipschitz maps $\theta_x : Fx \too Gx$ for all $x \in X$ satisfying that $G_{xx'}\theta_x = \theta_{x'}F_{xx'}$ for all $x, x' \in X$. We can define the composition of metric transformations $\theta$ and $\theta'$ by $(\theta'\theta)_x = \theta'_x\theta_x$. We denote the category of metric actions $X \too \Met$ and metric transformations by $\Met_X$.
\end{enumerate}
\end{df}

\begin{df}\label{metfibdef}
\begin{enumerate}
\item Let $\pi : E \too X$ be a Lipschitz map between metric spaces. We say that $\pi$ is a {\it metric fibration over} $X$ if it satisfies the following : for all $\e \in E$ and $x \in X$, there uniquely exists $\e_x \in \pi^{-1}x$ such that
\begin{itemize}
\item $d_E(\e, \e_x) = d_X(\pi \e, x)$,
\item $d_E(\e, \e') = d_E(\e, \e_x) + d_E(\e_x, \e')$ for all $\e' \in \pi^{-1}x$.
\end{itemize}
We call the point $\e_x$ the {\it lift of $x$ along $\e$}.
\item For metric fibrations $\pi : E \too X$ and $\pi' : E' \too X$, a {\it morphism} $\varphi : \pi \too \pi'$ is a Lipschitz map $\varphi : E \too E'$ such that $\pi'\varphi = \pi$. We denote the category of metric fibrations over $X$ and their morphisms by $\Fib_X$.
\end{enumerate}
\end{df}

\begin{eg}
For a product of metric spaces $E = X\times Y$, the projection $X\times Y \too X$ is a metric fibration. We call it a {\it trivial metric fibration}.
\end{eg}

\begin{lem}\label{liftfunct}
Let $\pi : E \too X$ be a metric fibration, and $x, x' \in X$. Then the correspondence $\pi^{-1}x \ni a \mapsto a_{x'} \in \pi^{-1}x'$ is an isometry, where we equip the sets $\pi^{-1}x$ and $\pi^{-1}x'$ with the induced metric from $E$.
\end{lem}
\begin{proof}
Note that the statement is obviously true if $E = \emptyset$. We suppose that $E \neq \emptyset$ in the following ; then every fiber $\pi^{-1}x$ is non-empty. For $a \in \pi^{-1}x$, we have $d_E(a_{x'}, a) = d_E(a_{x'}, (a_{x'})_x) + d_E((a_{x'})_x, a) = d_X(x', x) + d_E((a_{x'})_x, a)$. We also have $d_E(a, a_{x'}) = d_X(x, x')$. Hence we obtain that $d_E((a_{x'})_x, a) = 0$, and thus $(a_{x'})_x = a$ for all $x, x' \in X$. This implies that the correspondence is a bijection. Further, we have 
\[
d_E(a, b_{x'}) = d_E(a, a_{x'}) + d_E(a_{x'}, b_{x'}) = d_X(x, x') + d_E(a_{x'}, b_{x'})
\]
and
\[
d_E(b_{x'}, a) = d_E(b_{x'}, b) + d_E(b, a) = d_X(x', x) + d_E(b, a)
\]
for all $a, b \in \pi^{-1}x$. We thereby obtain that $d_E(a, b) = d_E(a_{x'}, b_{x'})$ for all $x, x' \in X$ and $a, b \in \pi^{-1}x$, which implies that the correspondence is an isometry. 
\end{proof}
\begin{lem}\label{morphcomm}
Let $\varphi : \pi \too \pi'$ be a morphism of metric fibrations. For all $x, x' \in X$ and $a \in \pi^{-1}x$, we have $(\varphi a)_{x'} = \varphi a_{x'}$.
\end{lem}
\begin{proof}
We have 
\begin{align*}
d_{E'}((\varphi a)_{x'}, \varphi a_{x'}) &= d_{E'}(\varphi a, \varphi a_{x'}) - d_X(x, x') \\
&\leq d_{E}(a,  a_{x'}) - d_X(x, x') \\
&= 0, 
\end{align*}
hence we obtain that $(\varphi a)_{x'} = \varphi a_{x'}$. 
\end{proof}

\begin{df}
Let $F : X \too \Met$ be a metric action. We define a metric fibration $\pi_F : E(F) \too X$ as follows : 
\begin{enumerate}
\item $E(F) = \{(x, a) \mid a \in Fx, x \in X\}$,
\item $d_{E(F)}((x, a), (x', b)) = d_X(x, x') + d_{Fx'}(F_{xx'}a, b)$,
\item $\pi_F(x, a) = x$.
\end{enumerate}
We call the above construction the {\it Grothendieck construction}.
\end{df}
\begin{prop}\label{grofun}
The Grothendieck construction gives a functor $ E : \Met_X \too \Fib_X$.
\end{prop}
\begin{proof}
Let $\theta : F \Longrightarrow G$ be a metric transformation. Then we construct Lipschitz maps $\varphi_\theta : E(F) \too E(G)$ by $\varphi_\theta (x, a) = (x, \theta_x a)$ for all $x \in X$ and $a \in Fx$. To see that $\varphi_\theta$ is a Lipschitz map, observe that
\begin{align*}
 d_{E(G)}(\varphi_\theta (x, a), \varphi_\theta (x', b)) &= 
d_{E(G)}((x, \theta_x a), (x', \theta_{x'} b)) \\
&= d_X(x, x') + d_{Gx'}(G_{xx'}\theta_x a, \theta_{x'}b) \\
&= d_X(x, x') + d_{Gx'}(\theta_{x'} F_{xx'} a, \theta_{x'}b) \\
&\leq  d_X(x, x') + d_{Fx'}( F_{xx'} a, b) \\
&=  d_{E(F)}((x, a), (x', b)).
\end{align*}
It remains to see that the correspondence $\theta \mapsto \varphi_\theta$ is functorial---that is we have $\varphi_{{\rm id}_F} = {\rm id}_{E(F)}$ and $\varphi_{\theta'\theta} = \varphi_{\theta'}\varphi_\theta$ for all metric transformations $\theta : F \Longrightarrow G$ and $\theta' : G \Longrightarrow H$. The former is obvious and the latter is checked as follows :
\begin{align*}
\varphi_{\theta'\theta}(x, a) &= (x, (\theta'\theta)_x a) \\
&= (x, \theta'_x\theta_x a) \\
&=  \varphi_{\theta'}\varphi_\theta(x, a).
\end{align*}
Finally, $\varphi_\theta$ is obviously a morphism of the metric fibration. 
\end{proof}
\begin{prop}\label{grofuninv}
We have a functor $F : \Fib_X \too \Met_X$ sending a metric fibration $\pi$ to a metric action $F_\pi$ with $F_\pi x = \pi^{-1}x$.
\end{prop}
\begin{proof}
Let $\pi : E \too X$ be a metric fibration. We define a metric action $F_\pi : X \too \Met$ by $F_\pi x = \pi^{-1}x$ and $(F_\pi)_{xx'}a = a_{x'}$ for all $x, x' \in X$ and $a\in \pi^{-1}x$, where we equip the set $\pi^{-1}x$ with the induced metric from $E$. It follows that $(F_\pi)_{xx} = {\rm id}_{F_\pi x}$ by the uniqueness of the lifts, and that $(F_\pi)_{xx'}$ defines an isometry $F_\pi x \too F_\pi x'$ with $(F_\pi)_{xx'}^{-1} = (F_\pi)_{x'x}$ by Lemma \ref{liftfunct}. Further, we have that
\begin{align*}
d_{F_\pi x''}((F_\pi)_{x'x''}(F_\pi)_{xx'}a, (F_\pi)_{xx''}a) &= d_{F_\pi x''}((a_{x'})_{x''}, a_{x''}) \\
&= d_{E}(a, (a_{x'})_{x''}) - d_X(x, x'') \\
&\leq  d_{E}(a, a_{x'}) + d_{E}(a_{x'}, (a_{x'})_{x''}) - d_X(x, x'') \\
&=  d_{X}(x, x') + d_{X}(x', x'') - d_X(x, x''), \\
\end{align*}
for all $x, x', x'' \in X$ and $a \in F_\pi x$. Hence $F_\pi$  defines a metric action $X \too \Met$. Next, let $\varphi : \pi \too \pi'$ be a morphism of metric fibrations. We define a metric transformation $\theta_\varphi : F_{\pi} \Longrightarrow F_{\pi'}$ by $(\theta_\varphi)_x a = \varphi a$ for all $x \in X$ and $a \in F_{\pi}x$. Then we have that 
\begin{align*}
(F_{\pi'})_{xx'}(\theta_\varphi)_x a &= (F_{\pi'})_{xx'}\varphi a \\
&= (\varphi a)_{x'} \\
&= \varphi a_{x'} \\
&= (\theta_\varphi)_{x'}(F_{\pi})_{xx'},
\end{align*}
where the third line follows from Lemma \ref{morphcomm}. Thus, $\theta_\varphi$  defines a metric transformation $F_{\pi} \Longrightarrow F_{\pi'}$. Note that we have $\theta_{{\rm id}_\pi} = {\rm id}_{F_\pi}$ and $(\theta_{\psi\varphi})_xa = \psi\varphi a = (\theta_{\psi})_x(\theta_{\varphi})_xa$ for morphisms $\varphi$ and $\psi$, which implies the functoriality of $F$. 
\end{proof}
The following is the counterpart of the correspondence between lax functors and the Grothendieck fibrations (B1 \cite{JPT}), and enhances Corollary 5.26 of \cite{A0}.
\begin{prop}\label{metfib}
The Grothendieck construction functor $ E : \Met_X \too \Fib_X$ is an  equivalence of categories.
\end{prop}
\begin{proof}
We show that $FE \cong {\rm id}_{\Met_X}$ and $EF \cong {\rm id}_{\Fib_X}$. It is immediate to verify $FE \cong {\rm id}_{\Met_X}$ by the definition. We show that $EF_\pi \cong \pi$ for a metric fibration $\pi : E \too X$. Note that $EF_\pi$ is a metric space consisting of points $(x, a)$ with $x \in X$ and $a \in \pi^{-1}x$, and we have $d_{EF_\pi}((x, a), (x', a')) = d_X(x, x') + d_{\pi^{-1}x'}(a_{x'}, a')$. We define a map $f : EF_\pi \too E$ by $f(x, a) = a$ for all $x \in X$ and $a \in \pi^{-1}x$. Then it is obviously an isometry and preserves fibers, hence is an isomorphism of metric fibrations. The naturality of this isomorphism is obvious. 

\end{proof}
\begin{rem}
Note that the trivial metric fibration corresponds to the constant metric action, that is $F_{xx'}={\rm id}$ for all $x, x' \in X$.
\end{rem}
\section{The  metric fundamental group of a metric space}
In this section, we give a concise introduction to {\it metric groups}. We also give a definition of {\it metric fundamental group}, which plays a role of $\pi_1$ for metric space in the classification of metric fibrations.
\subsection{Metric groups}
\begin{df}[cf. \cite{Roff2} Definition 6.1]\label{metgrpdef}
\begin{enumerate}
\item A {\it  metric group} is a monoid object in $\Met$ that is a group when we forget the metric space structure. That is, a metric space $\G$ equipped with a Lipschitz map $\cdot : \G\times\G \too \G$, a function $(-)^{-1} : \G \too \G$ and a point $e \in \G$ satisfying the suitable conditions of monoids and groups.
\item For metric groups $\mathcal{G}$ and $\mathcal{H}$, a {\it homomorphism} from $\G$ to $\H$ is a Lipschitz map $\G \too \H$ that commutes with the group structure.
\item We denote the category of metric groups and homomorphisms by $\MGrp$. 
\end{enumerate}
\end{df}
\begin{lem}\label{metgrp}
Let $(\G, d)$ be a metric group. Then 
\begin{enumerate}
\item we have $d(kg, kh) = d(g, h) = d(gk, hk)$ for all $g, h, k \in \G$.
\item we have $d(g, h) = d(g^{-1}, h^{-1})$ for all $g, h \in \G$.
\end{enumerate}
\end{lem}
\begin{proof}
\begin{enumerate}
\item Since the map $\G \too \G : g \too kg$ is a Lipschitz map for every $k \in \G$, we have $d(kg, kh) \leq d(g, h)$ and $d(k^{-1}(kg), k^{-1}(kh)) \leq d(kg, kh)$. Hence we obtain that $d(kg, kh) = d(g, h)$. The other identity can be proved similarly.

\item By (1), we have $d(g^{-1}, h^{-1}) = d(e, gh^{-1}) = d(h, g) = d(g, h)$.
\end{enumerate}

\end{proof}
\begin{eg}\label{auty}
Let $(X, d)$ be a metric space, and let $\Aut^u X$ be the set of isometries $f$ on $X$ such that $\sup_{x\in X}d_X(x, fx)< \infty$. We equip $\Aut^u X$ with a group structure by compositions. We also define a distance function on $\Aut^u X$ by $d_{\Aut^u X}(f, g) = \sup_{x\in X} d_X(fx, gx)$. It is straightforward to verify that $(\Aut^u X, d_{\Aut^u X})$ is a metric group.  Note that, if the metric space $X$ is bounded, meaning that
\[
\sup_{x,x'\in X} d_X(x, x')< \infty,
\]
then the group $\Aut^u X$ consists of all isometries of $X$. In this case we denote it by $\Aut X$.
\end{eg}

\begin{df}
\begin{enumerate}
\item A {\it normed group} is a group $G$ equipped with a map $|-| : G \too \R_{\geq 0}$ called a norm, satisfying that 
\begin{itemize}
\item $|g| = 0$ if and only if $g = e$,
\item $|gh| \leq |g| + |h|$ for all $g, h \in G$.
\end{itemize}

\item A norm on $G$ is called {\it conjugation invariant} if it satisfies that $|h^{-1}gh| = |g|$ for all $g, h \in G$.
\item A norm on $G$ is called {\it inverse invariant} if it satisfies that $|g^{-1}| = |g|$ for all $g \in G$.
\item For normed groups $G$ and $H$, a {\it normed homomorphism} from $G$ to $H$ is a group homomorphism $\varphi : G \too H$ satisfying that $|\varphi g|\leq |g|$.
\item We denote the category of conjugation and inverse invariant normed groups and normed homomorphisms by $\NGrp_{\rm conj}^{-1}$. 
\end{enumerate}
\end{df}
\begin{prop}[E. Roff,  \cite{Roff} Chap. 6 ]\label{eroff}
The categories $\MGrp$ and $\NGrp_{\rm conj}^{-1}$ are equivalent. 
\end{prop}
\begin{proof}
Given a metric group $\G$, we can define a conjugation and inverse invariant normed group ${\sf N}\G$ by 
\begin{itemize}
\item ${\sf N}\G = \G$ as a group,
\item $|g| = d_{\G}(e, g)$ for all $g \in {\sf N}\G$.
\end{itemize}
Note that this construction is functorial. Conversely, we can define a metric group ${\sf M}G$ given a conjugation and inverse invariant normed group $G$ by 
\begin{itemize}
\item ${\sf M}G = G$ as a group,
\item $d_{{\sf M}G}(g, h) = |h^{-1}g|$.
\end{itemize}
This construction is also functorial. It is straightforward to verify that the compositions of these functors are naturally isomorphic to the identities.  
\end{proof}
\subsection{The  metric fundamental group}
\begin{df}
Let $X$ be a metric space and $x \in X$. 
\begin{enumerate}
\item For each $n \geq 0$, we define a set $P_n(X, x)$ by
\[
P_n(X, x) := \{(x, x_1, \dots, x_n, x) \in X^{n+2}\}.
\]
We also define that $P(X, x) := \bigcup_nP_n(X, x)$.
\item We define a connected graph $G(X, x)$ with the vertex set $P(X, x)$ as follows. For $u, v \in P(X, x)$, an unordered pair $\{u, v\}$ spans an edge if and only if it satisfies both of the following : 
\begin{itemize}
\item There is an $n \geq 0$ such that $u \in P_n(X, x)$ and $v \in P_{n+1}(X, x)$.
\item There is a $0 \leq j \leq n$ such that $u_i = v_i$ for $1 \leq i \leq j$ and $u_i = v_{i+1}$ for $j+1 \leq i \leq n$, where we have $u = (x, u_1, \dots, u_n, x)$ and $v = (x, v_1, \dots, v_{n+1}, x)$.
\end{itemize}
\item We equip the graph $G(X, x)$ with a weighted graph structure by defining a function $w_{G(X, x)}$ on edges by 
\[
w_{G(X, x)}\{u, v\} = \begin{cases} d_X(v_j, v_{j+1}) + d_X(v_{j+1}, v_{j+2}) - d_X(v_j, v_{j+2}) & v_j\neq v_{j+2}, \\ 0 & v_j = v_{j+2},\end{cases}
\]
where we use the notations in (2).
\item We denote the quasi-metric space obtained from the weighted graph $G(X, x)$ by $Q(X, x)$. We  denote the Kolmogorov quotient of $Q(X, x)$ by $\pi_1^{m}(X, x)$.
\end{enumerate}
\end{df}

\begin{lem}
Let $X$ be a metric space and $x \in X$.
\begin{enumerate}
\item The metric space $\pi^m_1(X, x)$ has a metric group structure with multiplication defined by the concatenation defined as 
\[
[(x, u_1, \dots, u_n, x)]\bullet [(x, v_1, \dots, v_k, x)] = [(x, u_1, \dots, u_n, v_1, \dots, v_k, x)].
\]
The unit is given by $[(x, x)] \in \pi^m_1(X, x)$.
\item for all $x' \in X$, we have an isomorphism $\pi^m_1(X, x) \cong \pi^m_1(X, x')$ given by 
\[
[(x, u_1, \dots, u_n, x)] \mapsto [(x', x, u_1, \dots, u_n, x, x')].
\]

\end{enumerate}
\end{lem}
\begin{proof}
\begin{enumerate}
\item We first show that concatenation makes the weighted graph $G(X,x)$ into a monoid object in $\wGrph_{\sf conn}$. Let $(u, v), (u', v') \in G(X, x)\times G(X, x)$, and suppose that $\{(u, v), (u', v')\}$ spans an edge. Then we have that $u = u'$ and $v \in P_n(X, x), v' \in P_{n+1}(X, x)$, or $v = v'$ and $u \in P_n(X, x), u' \in P_{n+1}(X, x)$  for some $n$. We also have that $w_{G(X, x)\times G(X, x)}\{(u, v), (u', v')\} = w_{G(X, x)}\{u, u'\} + w_{G(X, x)}\{v, v'\}$. Note that $\{u\bullet v, u'\bullet v'\}$ spans an edge in $G(X, x)$. Further, we have $w_{G(X, x)}\{u\bullet v, u'\bullet v'\} = w_{G(X, x)}\{u, u'\} + w_{G(X, x)}\{v, v'\}$. Hence the concatenation map $\bullet : G(X, x)\times G(X, x) \too G(X, x)$ is a weighted graph homomorphism. It is immediate to verify that the identity is the element $(x, x)$ and that the multiplication is associative. Thus the weighted graph $G(X, x)$ is a monoid object in $\wGrph_{\sf conn}$, and by Proposition \ref{monoidal}, $\pi^m_1(X, x)$ is a monoid object in $\Met$. 

Now we show that any element $[(x, x_0, \dots, x_n, x)]$ has an inverse, namely the element $[(x, x_n, \dots, x_0, x)]$. It is enough to show that \[d_{Q(X, x)}((x, x_n, \dots, x_0, x_0, \dots, x_n x), (x, x)) = 0.\] And indeed, it is obvious that the elements $(x, x_n, \dots, x_0, x_0, \dots, x_n x)$ and $(x, x)$ can be connected by a path that consists of edges with weight $0$ in $G(X, x)$, which implies the desired equality. 

\item It is straightforward. \qedhere
\end{enumerate}
\end{proof}
\begin{df}\label{pi1def}
Let $X$ be a metric space and $x \in X$. We call the metric group $\pi_1^m(X, x)$  the {\it  metric fundamental group of } $X$ with respect to the base point $x$. We sometimes omit the base point and denote it by $\pi_1^m(X)$.
\end{df}

\begin{rem}
As a group, $\pi_1^m(X)$ can be obtained as the fundamental group of a simplicial complex $S_X$ whose $n$-simplices are subsets $\{x_0, \dots, x_n\}\subset X$ such that any distinct 3 points $x_i, x_j, x_k$ satisfy that $|\Delta(x_i, x_j, x_k)| = 0$ (see Definition \ref{degdeg}).
\end{rem}

\begin{rem}
Our fundamental group $\pi_1^m(X)$ is not functorial with respect to Lipschitz maps. However, it is  functorial with respect to Lipschitz maps that preserve collinearity ( $|\Delta(x_i, x_j, x_k)| = 0$ ), including every embedding of metric spaces.
\end{rem}

\section{$\PMet_X^\G \simeq \Tor_X^\G \simeq \Hom (\pi^m_1(X, x_0), \G)$}
In this section, we introduce the notion of `principal $\G$-bundles' for metric spaces. We define it from two different viewpoints, namely as a metric action and as a metric fibration, which turn out to be equivalent. As a metric action, we call it a {\it $\G$-metric action}, and as a metric fibration, we call it a {\it $\G$-torsor}. Then we show that they are classified by the conjugation classes of homomorphisms $\pi^m_1(X, x_0) \too  \G$.
\subsection{$\PMet_X^\G \simeq \Tor_X^\G$}
\begin{df}\label{pmetdef}
Let $X$ be a metric space and $\G$ be a metric group.
\begin{enumerate}
\item A {\it $\G$-metric action} $F : X \too \Met$ is a metric action satisfying the following : 
\begin{itemize}
\item For all $x \in X$, $F_x = \G$.
\item For all $x, x' \in X$, $F_{xx'}$ is a left multiplication by some $f_{xx'} \in \G$.
\end{itemize}
\item Let $F, G : X \too \Met$ be $\G$-metric actions. A  {\it $\G$-metric transformation} $\theta : F \Longrightarrow G$ is a metric transformation such that each component $\theta_x : Fx \too Gx$ is a left multiplication by an element $\theta_x \in \G$. We denote the category of $\G$-metric actions $X \too \Met$ and $\G$-metric transformations by $\PMet_X^{\G}$.
\end{enumerate}
\end{df}
\begin{rem}
Obviously, $\PMet_X^{\G}$ is a subcategory of $\Met_X$ and is also a groupoid.
\end{rem}
\begin{df}
Let $G$ be a group and $X$ be a metric space. We say that $X$ is a {\it right $G$-torsor} if $G$ acts on $X$ from the right and satisfies the following : 
\begin{itemize}
\item It is free and transitive.
\item For every $g \in G$, the map $g : X \too X$ is an isometry.
\item For all $x, x' \in X$ and every $g \in G$, we have $d_X(x, xg) = d_X(x', x'g)$.
\end{itemize}
\end{df}
\begin{lem}\label{gact}
Let $(X, d_X)$ be a  metric space and $G$ be a group. Suppose that $X$ is a right $G$-torsor. Then there exist a distance function $d_G$ on $G$ and a metric group structure $\cdot_x$ on $X$ for each $x \in X$ such that the map 
\[
G \too X ; g \mapsto xg
\]
gives an isomorphism of metric groups $(G, d_G) \cong (X, \cdot_x)$. Furthermore, the unit of the metric group $(X, \cdot_x)$ is $x$.
\end{lem}
\begin{proof}
Fix a point $x \in X$. We define a map $d_G : G \times G \too \R_{\geq 0}$ by $d_G(f, g) = d_X(xf, xg)$, which is independent from the choice of $x \in X$. It is immediate to check that $(G, d_G)$ is a metric space. Further, we have
\begin{align*}
d_G(ff', gg') &= d_X(xff', xgg') \\
&\leq  d_X(xff', xgf') + d_X(xgf', xgg') \\
&\leq  d_X(xf, xg) + d_X(xf', xg')\\
&= d_G(f, g) + d_G(f', g'),
\end{align*}
and 
\begin{align*}
d_G(f^{-1}, g^{-1}) &= d_X(xf^{-1}, xg^{-1}) \\
&=  d_X(x, xg^{-1}f) \\
&=  d_X(xg, (xg)g^{-1}f)\\
&= d_X(xg, xf) \\
&= d_X(xf, xg) \\
&= d_G(f, g),
\end{align*}
for all $f, f', g, g' \in G$. Hence $(G, d_G)$ is a metric group. 

Now we define a map $G \too X$ by $g \mapsto xg$. This map is an isometry by the definition. Hence we can transfer the metric group structure on $G$ to $X$ by this map. With respect to this group structure $\cdot_x$ on $X$, we have $x\cdot_x x' = eg' = x'$ and $x'\cdot_x x = g'e = x'$, where we put $x' = xg'$. Hence $x \in X$ is the unit of the group $(X, \cdot_x)$. 
\end{proof}
\begin{df}\label{pfib}
Let $G$ be a group. A metric fibration $\pi : E \too X$ is a {\it  $G$-torsor over $X$} if it satisfies the following : 
\begin{itemize}
\item $G$ acts isometrically on $E$ from the right, and preserves each fiber of $\pi$.
\item Each fiber of $\pi$ is a right $G$-torsor with respect to the above action.
\end{itemize}
\end{df}
\begin{lem}\label{fiberind}
Let $\pi : E \too X$ be a $G$-torsor, and $x, x' \in X$. Then the metric group structures on $G$ induced from the fibers $\pi^{-1}x$ and $\pi^{-1}x'$  are identical.
\end{lem}
\begin{proof}
Note that, for all $\e \in \pi^{-1}x$ and $f \in \G$, we have 
\begin{align*}
d_E((\e f)_{x'}, \e_{x'}f) &= d_E(\e f, \e_{x'}f) - d_E(\e f, (\e f)_{x'}) \\
&= d_E(\e, \e_{x'}) - d_E(\e f, (\e f)_{x'}) \\
&= d_X(x, x') - d_{X}(x, x') \\
&= 0,
\end{align*}
hence we obtain that $(\e f)_{x'} =  \e_{x'}f$. Let $d_x$ and $d_{x'}$ be the distance function on $G$ induced from the fibers $\pi^{-1}x$ and $\pi^{-1}x'$ respectively. Explicitly, for $\e \in \pi^{-1}x$ and $f, g \in G$, we have $d_x(f, g) = d_E(\e f, \e g)$ and $d_{x'}(f, g) = d_E(\e_{x'}f, \e_{x'}g)$.  Therefore we obtain that 
\[
d_{x'}(f, g) = d_E(\e_{x'}f, \e_{x'}g) = d_E((\e f)_{x'}, (\e g)_{x'}) = d_E(\e f, \e g) = d_x(f, g),
\]
by Lemma \ref{liftfunct}. 
\end{proof}
For a $G$-torsor $\pi : E \too X$, we can consider the group $G$ as a metric group that is isometric to a fiber of $\pi$ by Lemma \ref{gact}. Further, such a  metric structure is independent of the choice of the fiber by Lemma \ref{fiberind}. Hence, in the following, we write `$\G$-torsors' instead of `$G$-torsors', where $\G$ denotes the group $G$ equipped with this metric structure.
\begin{df}\label{tordef}
Let $\pi : E \too X$ and $\pi' : E' \too X$ be $\G$-torsors. A {\it $\G$-morphism} $\varphi : \pi \too \pi'$ is a $G$-equivariant map $E \too E'$ that is also a morphism of metric fibrations. We denote the category of $\G$-torsors over $X$ and $\G$-morphisms by $\Tor_X^\G$.
\end{df}
\begin{rem}
We can show that any $\G$-morphism is an isomorphism as follows : Note that for all $\e \in E, x \in X$ and $g \in \G$, we have $d_E(\e, \e_xg) = d_X(\pi \e, x) + |g|$ by the definitions. Then the $\G$-equivariance of $\varphi$ and Lemma \ref{morphcomm} implies that 
\begin{align*}
d_{E'}(\varphi \e, \varphi (\e_xg)) &= d_{E'}(\varphi \e, (\varphi \e)_xg) \\
&= d_X(\pi' \varphi \e, x) + |g| \\
&= d_X(\pi \e, x) + |g| \\
&= d_E(\e, \e_xg),
\end{align*}
which says that $\varphi$ preserves distances. The invertibility of $\varphi$ is immediate from the $G$-equivariance.
\end{rem}
Next we show the equivalence of $\G$-metric actions and $\G$-torsors.

\begin{prop}\label{pfunct}
The Grothendieck construction determines a functor $ E : \PMet_X^\G \too \Tor_X^\G$.
\end{prop}
\begin{proof}
Let $F : X \too \Met$ be a $\G$-metric action. Let $E(F)$ be the metric fibration given by the Grothendieck construction. Note that we have $d_{E(F)}((x, g), (x', g')) = d_X(x, x') + d_\G(g_{xx'}g, g')$. We define a $\G$ action on $E(F)$ by $(x, g)h = (x, gh)$ for all $g, h \in \G$ and $x \in X$. This is obviously compatible with the projection, and also free and transitive on each fiber. We also have that
\begin{align*}
d_{E(F)}((x, g)h, (x', g')h) &= d_{E(F)}((x, gh), (x', g'h)) \\
&= d_X(x, x') + d_\G(g_{xx'}gh, g'h) \\
&= d_X(x, x') + d_\G(g_{xx'}g, g') \\
&= d_{E(F)}((x, g), (x', g')),
\end{align*}
hence it acts isometrically. Further, we have that
\begin{align*}
d_{E(F)}((x, g), (x, g)h) &= d_{E(F)}((x, g), (x, gh)) \\
&= d_\G(g, gh) \\
&= d_\G(e, h), 
\end{align*}
hence each fiber is a right $\G$-torsor. Therefore, we obtain that $E(F)$ is a $\G$-torsor. 

Now let $\theta : F \Longrightarrow F'$ be a $\G$-metric transformation. The Grothendieck construction gives a map $\varphi_\theta : E(F) \too E(F')$ by $\varphi_\theta (x, g) = (x, \theta_x g)$, which is a morphism of metric fibrations. It is checked that $\varphi_\theta$ is $\G$-equivariant as follows : 
\begin{align*}
 (\varphi_\theta (x, g))h = (x, \theta_x gh) = \varphi_\theta (x, gh).
\end{align*}
Hence it is a $\G$-morphism. 
\end{proof}

\begin{prop}\label{pessur}
We have a functor $ F : \Tor_X^{\G} \too \PMet_X^{\G}$ sending a $\G$-torsor $\pi$ to a $\G$-metric action $F_\pi$ with $F_\pi x = \pi^{-1}x$.
\end{prop}
\begin{proof}
Let $\pi : E \too X$ be a $\G$-torsor. We fix points $x_0 \in X$ and $\e \in \pi^{-1}x_0$. For each $x \in X$, we equip the set $\pi^{-1}x$ with a metric group structure isomorphic to $\G$ with the unit $\e_x$ by Lemma \ref{gact}. Hence we can identify each fiber with $\G$ by the map $g \mapsto \e_xg$. Now we put $(\e_x)_{x'} = \e_{x'}g_{xx'} \in \pi^{-1}x'$ for $x, x' \in X$ and $g_{xx'} \in \G$. Then, for each $h \in \G$, we have
\begin{align*}
d_X(x, x') &= d_E(\e_xh, (\e_xh)_{x'}) \\
&= d_E(\e_x, (\e_xh)_{x'}h^{-1}) \\
&= d_E(\e_x, \e_{x'}g_{xx'}) + d_E(\e_{x'}g_{xx'}, (\e_xh)_{x'}h^{-1})\\
&= d_X(x, x') + d_E(\e_{x'}g_{xx'}, (\e_xh)_{x'}h^{-1}),
\end{align*}
hence we obtain that $(\e_xh)_{x'} = \e_{x'}g_{xx'}h$. This implies that the map $\pi^{-1}x \too \pi^{-1}x'$  given by lifts $\e_xh \mapsto (\e_xh)_{x'}$ is the left multiplication by $g_{xx'}$ when we identify each fiber with $\G$ as above. Hence the functor $F$ gives a $\G$-metric action. 

Next, let $\varphi : \pi \too \pi'$ be a $\G$-morphism between $\G$-torsors $\pi : E \too X$ and $\pi' : E' \too X$. It induces a Lipschitz map $\varphi_x : \pi^{-1}x \too \pi'^{-1}x$. Since fibers $\pi^{-1}x$ and $\pi'^{-1}x$ are identified with $\G$ and $\varphi_x$ is $\G$-equivariant, we can identify $\varphi_x$ with the left multiplication by $\varphi_x\e_x$. This implies that the functor $F$ sends the $\G$-morphism $\varphi$ to a $\G$-metric transformation between $F_{\pi}$ and $F_{\pi'}$. 
\end{proof}

\begin{prop}\label{pmettor}
The Grothendieck construction functor $\PMet_X^\G \too \Tor_X^\G$ is an  equivalence of categories.
\end{prop}
\begin{proof}
It is similar to the proof of Proposition \ref{metfib}.
\end{proof}
\subsection{$\PMet_X^\G \simeq \Hom (\pi^m_1(X, x_0), \G)$}\label{homcla}
First we define the category of homomorphisms of metric groups $\G \too  \G'$.
\begin{df}\label{homcat}
Let $\G$ and $\G'$ be metric groups, and let $\Hom (\G, \G')$ be the set of all homomorphisms $\G \too  \G'$. We equip $\Hom (\G, \G')$ with a groupoid structure by defining $\Hom (\G, \G')(\varphi, \psi) = \{h \in \G' \mid \varphi = h^{-1}\psi h\}$ for homomorphisms $\varphi, \psi : \G \too \G'$. The identity on $\varphi \in \Hom (\G, \G')$ is the unit $e \in \G'$, and the composition of morphisms $h \in \Hom (\G, \G')(\varphi, \psi)$ and $h' \in \Hom (\G, \G')(\psi, \xi)$ is defined by $h'\circ h = h'h$.
\end{df}
\begin{lem}\label{funcA}
Let $X$ be a metric space and $\G$ be a metric group. For each $x_0 \in X$, we have a functor $A : \Hom (\pi^m_1(X, x_0), \G) \too \PMet^{\G}_X$ sending a homomorphism $\varphi : \pi^m_1(X, x_0) \too \G$ to a $\G$-metric action $F_\varphi$ with $F_\varphi x = \G$.
\end{lem}
\begin{proof}
Let $\varphi : \pi^m_1(X, x_0) \too \G$ be a homomorphism. We define a $\G$-metric action $F_\varphi : X \too \Met$ by $F_\varphi x = \G$ and $(F_\varphi)_{xx'} = \varphi[(x_0, x', x, x_0)]\cdot : \G \too \G$, where we denote the left multiplication by $(-)\cdot$. It is verified that this defines a $\G$-metric action as follows. 

For all $x, x' \in X$, we have $(F_\varphi)_{xx} = \varphi[(x_0, x, x, x_0)]\cdot = e\cdot = {\rm id}_{\G}$, and $(F_\varphi)_{x'x} = \varphi[(x_0, x, x', x_0)]\cdot = (\varphi[(x_0, x', x, x_0)])^{-1}\cdot = (F_\varphi)_{xx'}^{-1}$. Further, we have 
\begin{align*}
d_\G((F_\varphi)_{x'x''}(F_\varphi)_{xx'}g, (F_\varphi)_{xx''}g) &= d_\G(\varphi[(x_0, x'', x', x_0)]\varphi[(x_0, x', x, x_0)], \varphi[(x_0, x'', x, x_0)]) \\
&= d_\G(\varphi[(x_0, x'', x', x, x_0)], \varphi[(x_0, x'', x, x_0)]) \\
&= d_\G((\varphi[(x_0, x'', x, x_0)])^{-1}\varphi[(x_0, x'', x', x, x_0)], e) \\
&= d_\G(\varphi[(x_0, x, x'', x', x, x_0)], e) \\
&\leq  d_{\pi^m_1(X, x_0)}([(x_0, x, x'', x', x, x_0)], [x_0, x_0]) \\
&\leq d_X(x, x') + d_X(x', x'') - d_X(x, x''),
\end{align*}
for all $x, x', x'' \in X$ and $g\in \G$. Let $h : \varphi \too \psi$ be a morphism in $\Hom (\pi^m_1(X, x_0), \G)$; that is, we have $\varphi = h^{-1}\psi h$ with $h \in \G$. Then we can construct a $\G$-metric transformation $\theta : F_\varphi \Longrightarrow F_\psi$ by $\theta_x = h\cdot : \G \too \G$. It satisfies that $(F_\psi)_{xx'}\theta_x = \theta_{x'}(F_\varphi)_{xx'}$ since we have $\psi[(x_0, x', x, x_0)]h = h\varphi[(x_0, x', x, x_0)]$. 
\end{proof}
\begin{lem}
Let $X$ be a metric space and $\G$ be a metric group. For each $x_0 \in X$, there is a functor $B : \PMet^{\G}_X \too \Hom (\pi^m_1(X, x_0), \G)$ sending a $\G$-metric action $F$ to a homomorphism $\varphi_F : \pi^m_1(X, x_0) \too  \G$ defined by 
\[
\varphi_F [(x_0, x_1, \dots, x_n, x_0)] = F_{x_1x_0}F_{x_2x_1}\dots F_{x_nx_{n-1}}F_{x_0x_n},
\]
for each $[(x_0, x_1, \dots, x_n, x_0)] \in \pi^m_1(X, x_0)$.
\end{lem}
\begin{proof}
 It is immediate to check that this is well defined. Let $F, F' : X \too \Met$ be $\G$-metric actions and  $\theta : F \Longrightarrow F'$ be a $\G$-metric transformation. Then we have 
\begin{align*}
\theta_{x_0}^{-1}\varphi_{F'}[(x_0, x_1, \dots, x_n, x_0)]\theta_{x_0} &= \theta_{x_0}^{-1}F'_{x_1x_0}F'_{x_2x_1}\dots F'_{x_{n}x_{n-1}}F'_{x_0x_n}\theta_{x_0} \\
&= F_{x_1x_0}F_{x_2x_1}\dots F_{x_nx_{n-1}}F_{x_0x_n}\\
&= \varphi_{F}[(x_0, x_1, \dots, x_n, x_0)],
\end{align*}
for each $[(x_0, x_1, \dots, x_n, x_0)] \in \pi^m_1(X, x_0)$. Hence $\theta_{x_0} \in \G$ gives a morphism $\theta_{x_0} : \varphi_F \too \varphi_{F'}$. This correspondence is obviously functorial. 
\end{proof}
\begin{prop}\label{hompri}
The functor $ A : \Hom (\pi^m_1(X, x_0), \G) \too \PMet^{\G}_X$ of Lemma \ref{funcA} is an  equivalence of categories.
\end{prop}
\begin{proof}
We show the natural isomorphisms $BA \cong {\rm id}_{\Hom (\pi^m_1(X, x_0), \G)}$ and $AB \cong {\rm id}_{\PMet^{\G}_X}$. For a homomorphism $\varphi : \pi^m_1(X, x_0) \too \G$, we have 
\begin{align*}
\varphi_{F_\varphi}[(x_0, x_1, \dots, x_n, x_0)] &= (F_\varphi)_{x_1x_0}(F_\varphi)_{x_2x_1}\dots (F_\varphi)_{x_0x_n} \\
&= \varphi[(x_0, x_0, x_1, x_0)]\varphi[(x_0, x_1, x_2, x_0)]\dots \varphi[(x_0, x_{n}, x_0, x_0)] \\
&= \varphi[(x_0, x_0, x_1, x_1,  \dots  x_n, x_n, x_0, x_0)] \\
&= \varphi[(x_0, x_1,  \dots  x_n, x_0)], 
\end{align*}
for each $[(x_0, x_1, \dots, x_n, x_0)] \in \pi^m_1(X, x_0)$. Hence we obtain an isomorphism $BA \varphi = \varphi$ that is obviously natural. Conversely, let $F : X \too \Met$ be a $\G$-metric action. Then we have $(F_{\varphi_F})_x = \G$ and $(F_{\varphi_F})_{xx'} = \varphi_F[(x_0, x', x, x_0)] = F_{x'x_0}F_{xx'}F_{x_0x}$. Now we define a $\G$-metric transformation $\theta : F_{\varphi_F} \Longrightarrow F$ by $\theta_x = F_{x_0x}$. It is obvious that we have $F_{xx'}\theta_x=\theta_{x'}(F_{\varphi_F})_{xx'}$, hence it is well defined and obviously an isomorphism. For a $\G$-metric transformation $\tau : F \Longrightarrow F'$, we have $(AB\tau)_x = \tau_{x_0}\cdot : (F_{\varphi_F})_{x} \too (F'_{\varphi_{F'}})_{x}$ by the construction. Hence the condition $\tau_xF_{x_0x} = F'_{x_0x}\tau_{x_0}$ of the $\G$-metric transformation implies  the naturality of this isomorphism. 
\end{proof}

\subsection{Example}
We give the following example of  metric fundamental group.
\begin{prop}\label{cyclegraph}
Let $C_n$ be the undirected $n$-cycle graph. Then we have
\[
\pi^m_1(C_n) \cong \begin{cases}\Z \text{ with } |1| = 1 & n : \text{odd}, \\ 0 & n : \text{even}. \end{cases}
\]
Hence we have that $\PMet_{C_n}^{\G} \simeq \begin{cases}\Hom (\Z, \G)  & n : \text{odd}, \\ 0 & n : \text{even}, \end{cases}$ for all metric group $\G$, which implies that there is only a trivial metric fibration over $C_{2n}$ and that there is at most one non-trivial metric fibration over $C_{2n+1}$.
\end{prop}
\begin{proof}
Let $V(C_n) = \{v_1, \dots, v_n\}$ be the vertex set whose numbering is anti-clockwise. For $C_{2n}$, it reduces to show that  $[(v_1, v_2, \dots, v_{2n}, v_1)] = [(v_1, v_1)]$. Since we have $d_{C_{2n}}(v_i, v_j) = d_{C_{2n}}(v_i, v_k) + d_{C_{2n}}(v_k, v_j)$ for all $i\leq k \leq j$ with $j-i\leq n$, we obtain that
\begin{align*}
[(v_1, v_2, \dots, v_{2n}, v_1)] &= [(v_1, \dots, v_{n+1}, \dots, v_{2n}, v_1)] \\
&=  [(v_1, v_{n+1},  v_1)] \\
&=  [(v_1, v_1)].
\end{align*}
For $C_{2n+1}$, the possible non-trivial element of $\pi^m_1(C_{2n+1})$ is a concatenation or its inverse of the element $[(v_1, \dots, v_{2n+1}, v_1)]$. Now we have $[(v_1, \dots, v_{2n+1}, v_1)] = [(v_1, v_{n+1}, v_{n+2}, v_1)]$, by the same argument as above, and 
\begin{align*}
&d_{Q(C_{2n+1}, v_1)}((v_1, v_{n+1}, v_{n+2}, v_1), (v_1, v_{n+1}, v_1)) \\
&= d_{C_{2n+1}}(v_{n+1}, v_{n+2}) + d_{C_{2n+1}}(v_{n+2}, v_{1}) - d_{C_{2n+1}}(v_{n+1}, v_{1}) \\
&= d_{C_{2n+1}}(v_{n+1}, v_{n+2}) \\
&= 1.
\end{align*}
Hence we obtain that $|[(v_1, \dots, v_{2n+1}, v_1)]| = 1$. 
\end{proof}
\begin{rem}
Note that the cycle graph $C_n$ is a metric group $\Z/n\Z$ with $|1| = 1$. Hence the examples in Figure 1 are $\Z/2\Z$-torsors, which are classified by $\Hom(\Z, \Z/2\Z) \cong \Z/2\Z$.
\end{rem}
\section{Classification of metric fibrations}
In this section, we classify general metric fibrations by fixing the base and the fiber. It is analogous to that of topological fiber bundles, namely it reduces to classifying principal bundles whose fiber is the structure group of the concerned fibration. We divide it into two cases, whether the fiber is bounded or not, since we need to consider {\it expanded metric spaces} for the unbounded case.
\subsection{The functor $\widehat{(-)}^{x_0}$}
Before we show the classification, we introduce a technical functor that will be used later. 
\begin{df}
For all metric action $F : X \too \Met$ and a point $x_0 \in X$, we define a metric action $\hF^{x_0} : X \too \Met$ as follows. We define that $\hF^{x_0} x = Fx_0$ and $\hF^{x_0}_{xx'} = F_{x'x_0}F_{xx'}F_{x_0x} : Fx_0 \too Fx_0$ for all $x, x' \in X$. Then it is verified that this defines a metric action as follows. We have $\hF^{x_0}_{xx} = F_{xx_0}F_{xx}F_{x_0x} = {\rm id}_{Fx_0} = {\rm id}_{\hF^{x_0} x}$. We also have $(\hF^{x_0}_{x'x})^{-1} = (F_{xx_0}F_{x'x}F_{x_0x'})^{-1} = F_{x'x_0}F_{xx'}F_{x_0x} = \hF^{x_0}_{xx'}$ and 
\begin{align*}
d_{\hF^{x_0} x''}(\hF^{x_0}_{x'x''}\hF^{x_0}_{xx'}a, \hF^{x_0}_{xx''}a) &= d_{Fx_0}(F_{x''x_0}F_{x'x''}F_{x_0x'}F_{x'x_0}F_{xx'}F_{x_0x}a, F_{x''x_0}F_{xx''}F_{x_0x}a) \\
&= d_{Fx''}(F_{x'x''}F_{xx'}F_{x_0x}a, F_{xx''}F_{x_0x}a) \\
&\leq d_X(x, x') + d_X(x', x'') - d_X(x, x''),
\end{align*}
for all $x, x', x'' \in X$ and $a \in \hF^{x_0} x$.
\end{df}
\begin{lem}\label{hatfff}
The correspondence $F \mapsto \hF^{x_0}$ defines a fully faithful functor $\widehat{(-)}^{x_0}: \Met_X \too \Met_X$. Further, this restricts to a fully faithful functor $\PMet_X^\G \too \PMet_X^\G$ for each metric group $\G$.
\end{lem}
\begin{proof}
Let $\theta : F \Longrightarrow G$ be a metric transformation. We define a metric transformation $\widehat{\theta}^{x_0} : \hF^{x_0} \Longrightarrow \widehat{G}^{x_0}$ by $\widehat{\theta}^{x_0}_x = \theta_{x_0} : \hF^{x_0}x \too \widehat{G}^{x_0}x ; a \mapsto \theta_{x_0}a$. Then we have 
\begin{align*}
\widehat{G}^{x_0}_{xx'}\widehat{\theta}^{x_0}_x &= G_{x'x_0}G_{xx'}G_{x_0x}\theta_{x_0} \\
&= G_{x'x_0}G_{xx'}\theta_{x}F_{x_0x} \\
&= G_{x'x_0}\theta_{x'}F_{xx'}F_{x_0x} \\
&= \theta_{x_0}F_{x'x_0}F_{xx'}F_{x_0x} \\
&= \widehat{\theta}^{x_0}_x\widehat{F}^{x_0}_{xx'},
\end{align*}
hence this defines a metric transformation. It is obvious that $\widehat{{\rm id}_F}^{x_0} = {\rm id}_{\hF^{x_0}}$ and $(\widehat{\theta' \theta})^{x_0} = \widehat{\theta}'^{x_0} \widehat{\theta}^{x_0}$. It is a faithful functor because $G_{xx_0}\theta_x = \theta_{x_0}F_{xx_0}$ implies that $\theta_{x} = \theta'_{x}$ for all $x \in X$ if two metric transformation $\theta, \theta'$ satisfies $\theta_{x_0} = \theta'_{x_0}$. By definition, it restricts to a faithful functor $\PMet_X^\G \too \PMet_X^\G$ for each metric group $\G$. Next we show the fullness. Let $\eta : \hF^{x_0} \Longrightarrow \widehat{G}^{x_0}$ be a metric transformation. Then we have $\widehat{G}^{x_0}_{x_0x}\eta_{x_0} = \eta_x\widehat{F}^{x_0}_{x_0x}$ and $\widehat{F}^{x_0}_{x_0x} = {\rm id}_{F_{x_0}}$, $\widehat{G}^{x_0}_{x_0x} = {\rm id}_{G_{x_0}}$. Hence we obtain that $\eta_{x_0} = \eta_x$ for all $x \in X$.

Now we define a metric transformation $\wt{\eta} : F \Longrightarrow G$ by $\wt{\eta}_x = G_{x_0x}\eta_{x_0}F_{xx_0} : Fx \too Gx$. Then we have 
\begin{align*}
G_{xx'}\wt{\eta}_x &= G_{xx'}G_{x_0x}\eta_{x_0}F_{xx_0} \\
&= G_{x_0x'}\widehat{G}^{x_0}_{xx'}\eta_{x}F_{xx_0} \\
&= G_{x_0x'}\eta_{x'}\widehat{F}^{x_0}_{xx'}F_{xx_0} \\
&= G_{x_0x'}\eta_{x'}F_{x'x_0}F_{xx'}F_{x_0x}F_{xx_0} \\
&= G_{x_0x'}\eta_{x_0}F_{x'x_0}F_{xx'} \\
&= \wt{\eta}_{x'}F_{xx'},
\end{align*}
hence this defines a metric transformation. We obviously have $\widehat{(\wt{\eta})}^{x_0} = \eta$, which implies that the functor $\widehat{(-)}^{x_0}$ is full. The restriction to $\PMet_X^\G \too \PMet_X^\G$ is immediate. 
\end{proof}
\begin{lem}
The functor $\widehat{(-)}^{x_0}: \Met_X \too \Met_X$ is split essentially surjective. Its restriction $\PMet_X^\G \too \PMet_X^\G$ is also split essentially surjective for all metric group $\G$.
\end{lem}
\begin{proof}
Let $F : X \too \Met$ be a metric action. We define a metric transformation $\theta : \hF^{x_0} \Longrightarrow F$ by $\theta_x = F_{x_0x} : \widehat{F}^{x_0}x \too Fx ; a \mapsto F_{x_0x}a$. It  satisfies that
\begin{align*}
F_{xx'}\theta_x &= F_{xx'}F_{x_0x} \\
&= F_{x_0x'}F_{x'x_0}F_{xx'}F_{x_0x} \\
&= \theta_{x'}\widehat{F}^{x_0}_{xx'}.
\end{align*}
Further, we define a metric transformation $\theta^{-1} : F \Longrightarrow \hF^{x_0}$ by $\theta^{-1}_x = F_{xx_0} : Fx \too \hF^{x_0}x$ for all $x \in X$. Then we have $\hF_{xx'}^{x_0}\theta^{-1}_x = \theta^{-1}_xF_{xx'}$ similarly to the above, hence it defines a metric transformation. It is obviously an isomorphism. The restriction to $\PMet_X^\G \too \PMet_X^\G$ is immediate. 

\end{proof}
\begin{cor}\label{singlefiber}
The functor $\widehat{(-)}^{x_0}: \Met_X \too \Met_X$ and its restriction $\PMet_X^\G \too \PMet_X^\G$ for all metric group $\G$ are category equivalences. \qed
\end{cor}
\begin{df}\label{fiby}
\begin{enumerate}
\item We denote the image of the functor $\widehat{(-)}^{x_0}: \Met_X \too \Met_X$ by $\wh{\Met}_X^{x_0}$.
\item For each metric space $Y$, we denote by $\Met_X^Y$ the full subcategory of $\Met_X$ that consists of metric actions $F : X \too \Met$ such that $Fx \cong Y$ for all $x \in X$. 
\item We denote the image of $\widehat{(-)}^{x_0}$ restricted to $\Met_X^Y$ and $\PMet_X^\G$ by $\wh{\Met}_X^{Y, x_0}$ and $\wh{\PMet}_X^{\G, x_0}$ respectively.
\item For each metric space $Y$, we denote by $\Fib_X^Y$ the full subcategory of $\Fib_X$ that consists of metric fibrations $\pi : E \too X$ such that $\pi^{-1}x \cong Y$ for all $x \in X$. 
\end{enumerate}
\end{df}
\begin{lem}\label{metfibpmet}
\begin{enumerate}
\item We have category equivalences $\Met_X^Y \too \wh{\Met}_X^{Y, x_0}$  and $\PMet_X^\G \too \wh{\PMet}_X^{\G, x_0}$.
\item The Grothendieck construction functor $E : \Met_X \too \Fib_X$ restricts to the category equivalence $\Met_X^Y \too \Fib_X^Y$.
\end{enumerate}
\end{lem}
\begin{proof}(1) follows from Corollary \ref{singlefiber}, and (2) follows from the proof of Proposition \ref{metfib}.
\end{proof}

\subsection{Classification for the case of bounded fibers}\label{bddfiber}

In this subsection, we suppose that $X$ and $Y$ are metric spaces and $Y$ is bounded. Note that in this case the group $\Aut Y$ of automorphisms is a metric group (Example \ref{auty}).
\begin{lem}
 We have a  faithful functor 
\[
-\curvearrowright Y : \PMet_X^{\Aut Y} \too \Met_X^{ Y}.
\]

\end{lem}
\begin{proof}
Let $F \in \PMet_X^{\Aut Y}$. We define a metric action $F \curvearrowright Y : X \too \Met$ by $(F \curvearrowright Y)x = Y$ and $(F \curvearrowright Y)_{xx'} = F_{xx'} : Y \too Y$. It is immediate to verify that this  defines a metric action. For an $\Aut Y$-metric transformation $\theta : F \Longrightarrow G$, we define a metric transformation $\theta \curvearrowright Y : F\curvearrowright Y \Longrightarrow G\curvearrowright Y$ by $(\theta \curvearrowright Y)_x = \theta_x : Y \too Y ; y \mapsto \theta_xy$. Then it is also immediate to verify that it is a metric transformation. Further, this obviously defines a faithful functor. 
\end{proof}
\begin{lem}
The functor $-\curvearrowright Y : \PMet_X^{\Aut Y} \too \Met_X^{Y}$ is split essentially surjective.
\end{lem}
\begin{proof}
Let $F \in \Met_X^{Y}$ and fix isometries $\varphi_x : Y \too Fx$ using the axiom of choice.  We define an $\Aut Y$-metric action $\Aut F$ by $(\Aut F)x = \Aut Y$ and $(\Aut F)_{xx'} = \varphi_{x'}^{-1}F_{xx'}\varphi_x\cdot$ that is a left multiplication. We can verify that it is an $\Aut Y$-metric action as follows. Note that we have $(\Aut F)_{xx} = \varphi_{x}^{-1}F_{xx}\varphi_x\cdot = {\rm id}_{\Aut Y}$ and $(\Aut F)_{xx'}^{-1} = \varphi_{x}^{-1}F_{x'x}\varphi_{x'}\cdot = (\Aut F)_{x'x}$. We also have that
\begin{align*}
&\ d_{\Aut Y}((\Aut F)_{x'x''}(\Aut F)_{xx'}, (\Aut F)_{xx''}) \\
&= d_{\Aut Y}(\varphi_{x''}^{-1}F_{x'x''}\varphi_{x'}\varphi_{x'}^{-1}F_{xx'}\varphi_x, \varphi_{x''}^{-1}F_{xx''}\varphi_x) \\
&= d_{\Aut Y}(\varphi_{x''}^{-1}F_{x'x''}F_{xx'}\varphi_x, \varphi_{x''}^{-1}F_{xx''}\varphi_x) \\
&= \sup_{a \in Y}d_{Y}(\varphi_{x''}^{-1}F_{x'x''}F_{xx'}\varphi_xa, \varphi_{x''}^{-1}F_{xx''}\varphi_xa) \\
&= \sup_{a \in Fx}d_{Fx''}(F_{x'x''}F_{xx'}a, F_{xx''}a) \\
&\leq d_X(x, x') + d_X(x', x'') - d_X(x, x'').
\end{align*}
Now we define a metric transformation $\varphi : \Aut F\curvearrowright Y \Longrightarrow F$ by $\varphi_x : (\Aut F\curvearrowright Y)x = Y \too Fx$. This  satisfies that $F_{xx'}\varphi_x = \varphi_{x'}(\Aut F\curvearrowright Y)_{xx'}$ and is an isomorphism by the definition.  
\end{proof}
Since the category $\PMet_X^{\Aut Y}$ is a groupoid, the image of the functor $-\curvearrowright Y$ is in ${\sf core}\Met_X^{Y}$. (Here ${\sf core}$ denotes the subcategory consisting of all isomorphisms, as in  \ref{catconv} (4).)
\begin{lem}
The functor $\wh{-\curvearrowright Y}^{x_0} : \PMet_X^{\Aut Y} \too {\sf core}\wh{\Met}_X^{Y, x_0}$ is full.
\end{lem}
\begin{proof}
Note that we have $\wh{-\curvearrowright Y}^{x_0} = \wh{(-)}^{x_0}\curvearrowright Y$ by the definitions. Since the functor $\wh{(-)}^{x_0} : \PMet_X^{\Aut Y} \too \PMet_X^{\Aut Y}$ is full by Lemma \ref{hatfff}, it will suffice to show that the restriction $-\curvearrowright Y : \wh{\PMet}_X^{\Aut Y, x_0} \too {\sf core}\wh{\Met}_X^{Y, x_0}$ is full. Let $\theta : \wh{F}^{x_0} \curvearrowright Y \Longrightarrow \wh{G}^{x_0} \curvearrowright Y$ be an isomorphism in $\wh{\Met}_X^{Y, x_0}$, where $F, G \in \PMet_X^{\Aut Y}$. Then we have an isometry $\theta_x : Y \too Y$ such that $G_{x'x_0}G_{xx'}G_{x_0x}\theta_x = \theta_{x'}F_{x'x_0}F_{xx'}F_{x_0x}$ for all $x, x' \in X$. Since we have $\theta_x \in \Aut Y$, we obtain a morphism $\theta' : \wh{F}^{x_0} \Longrightarrow \wh{G}^{x_0} \in \wh{\PMet}_X^{\Aut Y, x_0}$ defined by $\theta'_x = \theta_x$. It is obvious that we have $\theta' \curvearrowright Y = \theta$. 
 \end{proof}
\begin{cor}\label{changefiber}
The functor $\wh{-\curvearrowright Y}^{x_0} : \PMet_X^{\Aut Y} \too {\sf core}\wh{\Met}_X^{Y, x_0}$ is an  equivalence of categories. \qed
\end{cor}
\begin{cor}\label{pmetcore}
The categories $\PMet_X^{\Aut Y}$ and ${\sf core}\Fib_X^Y$ are equivalent.
\end{cor}
\begin{proof}
It follows from Corollary \ref{changefiber}, with ${\sf core}\Fib_X^Y \simeq {\sf core}\Met_X^Y \simeq  {\sf core}\wh{\Met}_X^{Y, x_0}$ by Lemma \ref{metfibpmet}.
\end{proof}

\subsection{Classification for the case of unbounded fibers}
To classify general metric fibrations, we generalize the discussions so far to {\it extended metric groups}.
\begin{df}\label{extmet}
\begin{enumerate}
\item An {\it extended metric group} is a monoid object in $\EMet$ that is a group when we forget the metric structure.
\item For extended metric groups $\mathcal{G}$ and $\mathcal{H}$, a {\it homomorphism} from $\G$ to $\H$ is a Lipschitz map $\G \too \H$ that commutes with the group structure.
\item We denote the category of extended metric groups and homomorphisms by $\EMGrp$. Note that the category $\MGrp$ is a full subcategory of $\EMGrp$.
\end{enumerate}
\end{df}

\begin{eg}
Let $(X, d)$ be a metric space, and let $\Aut X$ be the group of isometries on $X$. We define a distance function on $\Aut X$ by $d_{\Aut X}(f, g) = \sup_{x\in X} d_X(fx, gx)$. Then it is immediate to verify that $(\Aut X, d_{\Aut X})$ is an extended metric group. We note that the `unit component' of $\Aut X$, that is the set of isometries $f$ such that $d_{\Aut X}({\rm id}_X, f)< \infty$, is exactly $\Aut^u X$ (Example \ref{auty}). Note that, if the metric space $X$ has finite diameter, then we have $\Aut X = \Aut^u X$.
\end{eg}

\begin{df}
Let $\G$ and $\G'$ be extended metric groups, and let $\Hom (\G, \G')$ be the set of homomorphisms. We equip $\Hom (\G, \G')$ with a groupoid structure similarly to the metric group case by defining $\Hom (\G, \G')(\varphi, \psi) = \{h \in \G' \mid \varphi = h^{-1}\psi h\}$ for all homomorphisms $\varphi, \psi : \G \too \G'$. 
\end{df}

\begin{rem}
We note that the same statement as Lemma \ref{metgrp} holds for extended metric groups. Further, the relationship between extended metric spaces and normed groups similar to Proposition \ref{metgrp} holds if we replace the codomain of norms by $[0, \infty]$.
\end{rem}

\begin{df}
Let $\G$ be an extended metric group and $X$ be a metric space. An {\it extended $\G$-metric action} $F$ is a correspondence $X \ni x \mapsto Fx = \G$ and $F_{xx'} \in \G$ such that 
\begin{itemize}
\item $F_{xx} = e, F_{xx'} = F_{x'x}^{-1}$,
\item $d_\G(F_{x'x''}F_{xx'}, F_{xx''}) \leq d_X(x, x')+d_X(x', x'') - d_X(x, x'')$.
\end{itemize}
For extended $\G$-metric actions $F$ and $G$, an {\it extended $\G$-metric transformation} $\theta : F\Longrightarrow G$ is a family of elements $\{\theta_x \in \G\}_{x\in X}$ such that $G_{xx'}\theta_x = \theta_{x'}F_{xx'}$. We denote the category of extended $\G$-metric actions and extended $\G$-metric transformations by $\EPMet_X^\G$. 
\end{df}
The following is obtained from the same arguments in subsection \ref{homcla} by replacing the term `metric group' by `extended metric group'.
\begin{prop}
For an extended metric group $\G$ and a  metric space $X$,  the categories $\EPMet^{\G}_X$ and $  \Hom (\pi^m_1(X, x_0), \G)$ are equivalent. \qed
\end{prop}

Further, the arguments in subsection \ref{bddfiber} can be applied for extended case, and we obtain the following.
\begin{prop}\label{extpmetcore}
For all metric spaces $X$ and $Y$, the categories $\EPMet_X^{\Aut Y}$ and ${\sf core}\Fib_X^Y$ are equivalent. Hence metric fibrations with fiber $Y$ are classified by $ \Hom (\pi^m_1(X, x_0), \G)$. \qed
\end{prop}

\section{Cohomological interpretation}
In this section, we give a cohomological classification of $\G$-torsors. It is an analogy of the $1$-\v{C}ech cohomology. Before giving the definition, we introduce the following technical term.
\begin{df} \label{degdeg}
Let $X$ be a metric space, and $x_1, x_2, x_3 \in X$. We denote the subset $\{x_1, x_2, x_3\} \subset X$ by $\Delta(x_1, x_2, x_3)$ and call it a {\it triangle}. We define the {\it degeneracy degree of the triangle} $\Delta(x_1, x_2, x_3)$ by 
\[
|\Delta(x_1, x_2, x_3)| := \min \left\{d_X(x_i, x_j) + d_X(x_j, x_k) - d_X(x_i, x_k) \mid \{i, j, k\} = \{1, 2, 3\}\right\}.
\]
Note that it is enough to consider $i, j, k$'s running in the cyclic order to obtain the above minimum.
\end{df}
The following is the definition of our `$1$-\v{C}ech chomology'.
\begin{df}\label{cech}
Let $X$ be a metric space and suppose that points of $X$ are indexed as $X = \{x_i\}_{i \in I}$. For a metric group $\G$, we define the {\it $1$-cohomology of $X$ with coefficients in} $\G$ as the category $\Hc^1(X; \G)$ by
\[
\Ob\Hc^1(X; \G)  = \left\{(a_{ijk}) \in \G^{I^3} \mid a_{ijk}a_{kj\ell } = a_{ij\ell}, |a_{ijk}a_{jki}a_{kij}| \leq |\Delta(x_i, x_j, x_k)|\right\},
\]
and 
\[
 \Hc^1(X; \G)((a_{ijk}), (b_{ijk})) = \left\{(f_{ij}) \in \G^{I^2} \mid a_{ijk}f_{jk} = f_{ij}b_{ijk} \right\},
\]
where we denote the conjugation invariant norm on $\G$ by $|-|$. We call an object of $\Hc^1(X; \G)$ {\it a cocycle}. Obviously, the above constructions are independent from the choice of the index $I$.
\end{df}
\begin{rem}
Note that, for a cocycle $(a_{ijk}) \in \Hc^1(X; \G)$, the condition $a_{ijk}a_{kj\ell} = a_{ij\ell}$ implies that $a_{iji} = e$ and $a_{ijk} = a_{kji}^{-1}$ for all $i, j, k \in I$. Further, for a morphism $(f_{ij})$, we have $f_{ij} = f_{ji}$ from the condition $a_{ijk}f_{jk} = f_{ij}b_{ijk}$ and $a_{iji} = b_{iji} = e$.
\end{rem}
\begin{lem}
The  $1$-cohomology of $X$ with coefficients in $\G$ is well defined. That is, $\Hc^1(X; \G)$ is indeed a category, in fact a groupoid.
\end{lem}
\begin{proof}
Let $(a_{ijk}), (b_{ijk}), (c_{ijk}) \in \Ob\Hc^1(X; \G)$, and $(f_{ij}) : (a_{ijk}) \too  (b_{ijk})$ and $(f'_{ij}) : (b_{ijk}) \too  (c_{ijk})$ be morphisms. Then $(f'\circ f)_{ij} := f_{ij}f'_{ij}$ defines a morphism $((f'\circ f)_{ij}) : (a_{ijk}) \too  (c_{ijk})$ since we have 
\[
a_{ijk}f_{jk}f'_{jk} = f_{ij}b_{ijk}f'_{jk} = f_{ij}f'_{ij}c_{ijk}.
\]
It obviously satisfies the associativity. The identity on $a_{ijk}$ is obviously defined by $e_{ij} = e$, where $e$ denotes the unit of $\G$. Further, $(f^{-1}_{ij})$ defines a morphism $(f^{-1}_{ij}) : b_{ijk} \too a_{ijk}$ that is the inverse of $(f_{ij})$. 
\end{proof}

\begin{prop}\label{pasting}
We have a faithful functor $\beta : \Hc^1(X; \G) \too  \Tor^{\G}_X$.
\end{prop}
\begin{proof}
For $(a_{ijk}) \in \Ob \Hc^1(X; \G)$, we define a $\G$-torsor $\beta (a_{ijk})$ as follows. Let $\mathcal{U} = \coprod_{(i, j) \in I^2}\G_{ij}$, where $\G_{ij} = \G^{ij}_i \coprod \G^{ij}_j = \G \coprod \G$. We write an element of $\G^{ij}_\bullet$ as $g^{ij}_\bullet$ and we  denote the identification $\G = \G^{ij}_\bullet$ by the map $\G \too \G^{ij}_\bullet ; g \mapsto g^{ij}_\bullet$, where $\bullet \in \{i, j\}$ for all $i \neq j \in I$. We define an equivalence relation $\sim$ on $\mathcal{U}$ generated by 
\[
g^{ij}_j \sim (ga_{ijk})^{jk}_j.
\]
Note that we have $g^{ij}_j\sim g^{ji}_j$ for all $i, j \in I$. We denote the quotient set $\mathcal{U}/\sim$ by $\beta (a_{ijk})$. Then we have a surjective map $\pi : \beta (a_{ijk}) \too X$ defined by $\pi [g^{ij}_j] = x_j$. For this map $\pi$, we have the following. 
\begin{lem}\label{gfiber}
For all $i, j \in I$, the map $\G \too \pi^{-1}x_j ; g \mapsto [g^{ij}_j]$ is a bijection.
\end{lem}
\begin{proof}
The surjectivity is clear. We show the injectivity. Suppose that we have $[g^{ij}_j] = [h^{ij}_j]$ for $g, h \in \G$. That is, we have elements $a_{k_0jk_1}, a_{k_1jk_2},  \dots, a_{k_{N-1}jk_N} \in \G$ such that $ga_{k_0jk_1} \dots a_{k_{N-1}jk_N} = h$ and $k_0 = k_N = i$. Then the condition $a_{ijk}a_{kj\ell } = a_{ij\ell}$ implies that $ga_{iji}=h$, hence $g = h$. 
\end{proof}
Note that Lemma \ref{gfiber} ensures that $[g^{ij}_j] = [h^{jk}_j]$ implies $h = ga_{ijk}$. Now we can define a distance function $d_{\beta (a_{ijk})}$ on $\beta (a_{ijk})$ as follows. Let $\varepsilon_i \in \pi^{-1}x_i$ and $ \varepsilon_j \in \pi^{-1}x_j$. Then there uniquely exist $g, h \in \G$ such that $[g^{ij}_i] = \varepsilon_i$ and $[h^{ij}_j] = \varepsilon_j$ by Lemma \ref{gfiber}. Then we define that
\[
d_{\beta (a_{ijk})}(\varepsilon_i, \e_j) = d_X(x_i, x_j) + d_\G(g, h). 
\]
The non-degeneracy is clear. The symmetry follows from that $[g^{ij}_i] =  [g^{ji}_i]$. The triangle inequality is verified as follows. Let $\varepsilon_i \in \pi^{-1}x_i$, $\varepsilon_j \in \pi^{-1}x_j$ and $\varepsilon_k \in \pi^{-1}x_k$.  Suppose that we have $[g^{ij}_i] = \varepsilon_i = [g'^{ik}_i]$, $[h^{ij}_j] = \varepsilon_j = [h'^{jk}_j]$,  and $[m^{jk}_k] = \varepsilon_k = [m'^{ik}_k]$. Then we have $g = g'a_{kij}$, $h' = ha_{ijk}$ and $m = m'a_{ikj}$, hence we obtain that
\begin{align*}
&\ d_{\beta (a_{ijk})}(\varepsilon_i, \varepsilon_j) + d_{\beta (a_{ijk})}(\varepsilon_j, \varepsilon_k) \\
&= d_X(x_i, x_j) + d_\G(g, h) + d_X(x_j, x_k) + d_\G(h', m) \\
&= d_X(x_i, x_j) + d_X(x_j, x_k) + d_\G(g'a_{kij}, h) + d_\G(ha_{ijk}, m'a_{ikj}) \\
&= d_X(x_i, x_j) + d_X(x_j, x_k) + d_\G(g'a_{kij}, h) + d_\G(ha_{ijk}a_{jki}a_{kij}, m'a_{kij}) \\
&\ + d_\G(h, ha_{ijk}a_{jki}a_{kij}) - d_\G(h, ha_{ijk}a_{jki}a_{kij})  \\
&\geq d_X(x_i, x_j) + d_X(x_j, x_k) + d_\G(g'a_{kij}, m'a_{kij})  - |a_{ijk}a_{jki}a_{kij}|  \\
&\geq d_X(x_i, x_j) + d_X(x_j, x_k) + d_\G(g', m')  - |\Delta(x_i, x_j, x_k)|  \\
&\geq d_X(x_i, x_k)  + d_\G(g', m')   \\
&= d_{\beta (a_{ijk})}(\varepsilon_i, \varepsilon_k).
\end{align*}
Now a map $\pi : \beta (a_{ijk}) \too X$ is obviously a 1-Lipschitz map. We verify that it is a metric fibration as follows. Let $x_i, x_j \in X$ and $\varepsilon_i \in \pi^{-1}x_i$. Suppose that we have $\varepsilon_i = [g^{ij}_i]$ for $g \in \G$. Then $\varepsilon_j := [g^{ij}_j] \in \pi^{-1}x_j$ is the unique element in $\pi^{-1}x_j$ such that $d_{\beta (a_{ijk})}(\varepsilon_i, \varepsilon_j) = d_X(x_i, x_j)$. Also, for $\varepsilon'_j := [h^{ij}_j] \in \pi^{-1}x_j$, we have $d_{\beta (a_{ijk})}(\varepsilon_i, \varepsilon'_j) = d_X(x_i, x_j) + d_\G(g, h) = d_{\beta (a_{ijk})}(\varepsilon_i, \varepsilon_j) + d_{\beta (a_{ijk})}(\varepsilon_j, \varepsilon'_j)$. Finally, we equip the metric fibration $\pi : \beta (a_{ijk}) \too X$ with a right action by $\G$ as $[g^{ij}_\bullet]h = [(h^{-1}g)^{ij}_\bullet]$ for all $i, j \in I$ and $\bullet \in \{i, j\}$. This is well defined since we have that 
\[
[(ga_{ijk})^{jk}_j]h = [(h^{-1}ga_{ijk})^{jk}_j] = [(h^{-1}g)^{ij}_j] = [g^{ij}_j]h.
\]
It is straightforward to verify that this is a $\G$-torsor.

Next we show the functoriality. Let $(f_{ij}) : (a_{ijk}) \too (b_{ijk}) \in  \Hc^1(X; \G)$. We construct a map $f_\ast : \beta(a_{ijk}) \too \beta(b_{ijk})$ by $[g^{ij}_{\bullet}] \mapsto [(gf_{ij})^{ij}_{\bullet}]$ for all $i, j \in I$ and $\bullet \in \{i, j\}$. It is well defined since we have that
\[
[(ga_{ijk})^{jk}_j] \mapsto [(ga_{ijk}f_{jk})^{jk}_j]= [(gf_{ij}b_{ijk})^{jk}_{j}] = [(gf_{ij})^{ij}_{j}].
\]
The map $f_\ast$ obviously preserves fibers, and is an isometry since we have that 
\begin{align*}
d_{\beta(b_{ijk})}(f_\ast [g^{ij}_i], f_\ast [h^{ij}_j]) &= d_{\beta(b_{ijk})}( [(gf_{ij})^{ij}_i], [(hf_{ij})^{ij}_j]) \\
&= d_X(x_i, x_j) + d_\G(gf_{ij}, hf_{ij}) \\
&= d_X(x_i, x_j) + d_\G(g, h) \\
&= d_{\beta(a_{ijk})}([g^{ij}_i], [h^{ij}_j]).
\end{align*}
Further, it is $\G$-equivariant since we have that
\[
(f_\ast[g^{ij}_{j}])m = [(gf_{ij})^{ij}_{j}]m = [(m^{-1}gf_{ij})^{ij}_{j}] = f_\ast([g^{ij}_{j}]m).
\]
The faithfulness is obvious from the construction. 
\end{proof}
\begin{prop}
The functor $\beta : \Hc^1(X; \G) \too  \Tor^{\G}_X$ is full.
\end{prop}
\begin{proof}
Let $(a_{ijk}), (b_{ijk}) \in \Ob \Hc^1(X; \G)$ be cocycles, and suppose that we have a morphism $\varphi : \beta(a_{ijk}) \too \beta(b_{ijk})$ in $\Tor^{\G}_X$. We denote the projections $\beta(a_{ijk}) \too X$ and $\beta(b_{ijk}) \too X$ by $\pi_a$ and $\pi_b$ respectively. For all $i, j \in I$, we have bijections $A_{ij} : \G \too \pi_{a}^{-1}x_j$ and $B_{ij} : \G \too \pi_{b}^{-1}x_j$ given by $g \mapsto [g^{ij}_j]$ by Lemma \ref{gfiber}. Define a map $\varphi_{ij} = B_{ij}^{-1}\varphi A_{ij} : \G \too \G$. Note that we have $\varphi[g^{ij}_j] = [(\varphi_{ij}g)^{ij}_j]$. Now the $\G$-equivariance of $\varphi$ implies that 
\[
\varphi[g^{ij}_j] = \varphi[(ge)^{ij}_j] = (\varphi[e^{ij}_j])g^{-1} = [(\varphi_{ij}e)^{ij}_j]g^{-1} = [(g\varphi_{ij}e)^{ij}_j],
\]
which implies that $\varphi_{ij}g = g\varphi_{ij}e$ by Lemma \ref{gfiber}. From this, we obtain that
\[
\varphi[(ga_{ijk})^{jk}_j] = \varphi[(ga_{ijk})^{kj}_j] = [(\varphi_{kj}(ga_{ijk}))^{kj}_j] = [(ga_{ijk}\varphi_{kj}e)^{kj}_j].
\]
Since we have $[g^{ij}_j] = [(ga_{ijk})^{jk}_j]$, we obtain that  $a_{ijk}\varphi_{kj}e = (\varphi_{ij}e)b_{ijk}$ by Lemma \ref{gfiber}. Further, since the lift of $x_j$ along $[g^{ij}_i]$ is $[g^{ij}_j]$ and $\varphi$ preserves the lift, the conditions $\varphi[g^{ij}_j] = [(\varphi_{ij}g)^{ij}_j]$ and $\varphi[g^{ji}_i] = [(\varphi_{ji}g)^{ji}_i]$ implies that $\varphi_{ij} = \varphi_{ji}$. Hence we obtain a morphism $(\varphi_{ij}e) : (a_{ijk}) \too (b_{ijk})$ in $\Hc^1(X; \G)$, which satisfies that $\beta (\varphi_{ij}e) = \varphi$ by construction. 
\end{proof}
\begin{df}
Let $\pi : E \too X$ be a $\G$-torsor. For $x_i, x_j \in X$, we define a {\it local section of } $\pi$ {\it over a pair} $(x_i, x_j)$ to be a pair of points $(\varepsilon_i, \varepsilon_j) \in E^2$ such that $\pi \e_i = x_i, \pi \e_j = x_j$ and $\varepsilon_j$ is the lift of $x_j$ along $\varepsilon_i$.  We say that $((\varepsilon^{ij}_i,\varepsilon^{ij}_j))_{(i, j)\in I^2}$ is a {\it local section of} $\pi$ if each $(\varepsilon^{ij}_i, \varepsilon^{ij}_j)$ is a local section of $\pi$ over a pair $(x_i, x_j)$ and satisfies that $\varepsilon^{ij}_i = \varepsilon^{ji}_i$. 
\end{df}

\begin{prop}\label{alpha}
Let $\pi : E \too X$ be a $\G$-torsor. For a local section $s =((\varepsilon^{ij}_i,\varepsilon^{ij}_j))_{(i, j)\in I^2}$ of $\pi$, we can construct a cocycle $\alpha_s \pi \in \Ob \Hc^1(X;\G)$. Further, for all two local sections $s, s'$ of $\pi$, the corresponding cocycles $\alpha_s \pi$ and $\alpha_{s'} \pi$ are isomorphic.
\end{prop}
\begin{proof}
We define $a_{ijk} \in \G$ as the unique element such that $\varepsilon^{ij}_ja_{ijk} = \varepsilon^{jk}_j$. Then $(a_{ijk})$ satisfies that $a_{ijk}a_{kj\ell} = a_{ij\ell}$ since we have 
\[
\varepsilon^{ij}_ja_{ijk}a_{kj\ell} = \varepsilon^{jk}_ja_{kj\ell} = \varepsilon^{kj}_ja_{kj\ell} = \varepsilon^{j\ell}_j.
\]
Now note that we have $\varepsilon_xg = (\varepsilon g)_x$ for all $\varepsilon \in E, x \in X$ and $g \in \G$. Hence we have that 
\begin{align*}
\varepsilon^{ij}_ja_{ijk}a_{jki}a_{kij} &= \varepsilon^{jk}_ja_{jki}a_{kij} \\
&= (\varepsilon^{jk}_k)_{x_j}a_{jki}a_{kij} \\
&= (\varepsilon^{jk}_ka_{jki})_{x_j}a_{kij} \\
&= (\varepsilon^{ki}_k)_{x_j}a_{kij} \\
&= ((\varepsilon^{ki}_i)_{x_k}a_{kij})_{x_j} \\
&= ((\varepsilon^{ki}_ia_{kij})_{x_k})_{x_j} \\
&= ((\varepsilon^{ij}_i)_{x_k})_{x_j}.
\end{align*}
It follows that 
\begin{align*}
|a_{ijk}a_{jki}a_{kij}| &= d_E(\varepsilon^{ij}_j, \varepsilon^{ij}_ja_{ijk}a_{jki}a_{kij}) \\
&= d_E(\varepsilon^{ij}_j, ((\varepsilon^{ij}_i)_{x_k})_{x_j}) \\
&=  -d_E(\varepsilon^{ij}_j, \varepsilon^{ij}_i) + d_E(\varepsilon^{ij}_i, ((\varepsilon^{ij}_i)_{x_k})_{x_j}) \\
&\leq  -d_E(\varepsilon^{ij}_j, \varepsilon^{ij}_i) + d_E(\varepsilon^{ij}_i, (\varepsilon^{ij}_i)_{x_k}) + d_E((\varepsilon^{ij}_i)_{x_k}, ((\varepsilon^{ij}_i)_{x_k})_{x_j}) \\
&=  -d_X(x_j, x_i) + d_X(x_i, x_k) + d_X(x_k, x_j).
\end{align*}
Since the norm $|-|$ on $\G$ is conjugation invariant, the value $|a_{ijk}a_{jki}a_{kij}|$ is invariant under the cyclic permutation on $\{i, j, k\}$, hence we obtain that $|a_{ijk}a_{jki}a_{kij}| \leq |\Delta(x_i, x_j, x_k)|$. Thus we obtain a cocycle $\alpha_s \pi := (a_{ijk}) \in \Ob \Hc^1(X ; \G)$. Suppose that we have local sections $s = ((\varepsilon^{ij}_i,\varepsilon^{ij}_j))_{(i, j)\in I^2}$ and $s' = ((\mu^{ij}_i,\mu^{ij}_j))_{(i, j)\in I^2}$. Then there exists an element $(f_{ij}) \in \G^{I^2}$ such that $(\varepsilon^{ij}_if_{ij},\varepsilon^{ij}_jf_{ij}) = (\mu^{ij}_i,\mu^{ij}_j)$. Let $\alpha_s \pi = (a_{ijk})$ and $\alpha_{s'} \pi = (b_{ijk}) $. Then we obtain that 
\[
\varepsilon^{ij}_ja_{ijk}f_{jk}b^{-1}_{ijk} = \varepsilon^{jk}_jf_{jk}b^{-1}_{ijk} = \mu^{jk}_jb^{-1}_{ijk} = \mu^{ij}_j,
\]
which implies that $f_{ij} = a_{ijk}f_{jk}b^{-1}_{ijk}$. Hence $(f_{ij})$ defines a morphism $(f_{ij}) : (a_{ijk}) \too (b_{ijk})$ in $\Hc^1(X; \G)$. Since $\Hc^1(X; \G)$ is a groupoid, this is an isomorphism. 
\end{proof}
\begin{prop}\label{cohesssur}
The functor $\beta : \Hc^1(X; \G) \too  \Tor^{\G}_X$ is split essentially surjective.
\end{prop}
\begin{proof}
Let $\pi : E \too X$ be a $\G$-torsor.  Fix a local section $s = ((\varepsilon^{ij}_i,\varepsilon^{ij}_j))_{(i, j)\in I^2}$ of $\pi$. Let $\alpha_s\pi = (a_{ijk})$ be the cocycle constructed in Proposition \ref{alpha}. We show that the $\G$-torsors $\beta(a_{ijk})$ and $\pi$ are isomorphic. We define a map $\varphi : \beta(a_{ijk}) \too E$  by $[g^{ij}_\bullet] \mapsto \varepsilon^{ij}_\bullet g^{-1}$. It is well defined since we have that
\[
[(ga_{ijk})^{jk}_j] \mapsto \varepsilon^{jk}_ja^{-1}_{ijk}g^{-1} = \varepsilon^{ij}_jg^{-1}.
\]
It obviously preserves fibers and is a bijection. Also, it is an isometry since we have that
\begin{align*}
d_E(\varphi[g^{ij}_i], \varphi[h^{ij}_j]) &= d_E(\varepsilon^{ij}_ig^{-1}, \varepsilon^{ij}_jh^{-1}) \\
&= d_E(\varepsilon^{ij}_i, \varepsilon^{ij}_jh^{-1}g) \\
&= d_E(\varepsilon^{ij}_i, \varepsilon^{ij}_j) + d_E(\varepsilon^{ij}_j, \varepsilon^{ij}_jh^{-1}g) \\
&= d_X(x_i, x_j) + d_\G(g^{-1}, h^{-1}) \\
&= d_{\beta(a_{ijk})}([g^{ij}_i], [h^{ij}_j]).
\end{align*}
Further, it is immediately verified that $\varphi$ is $\G$-equivariant. Hence the map $\varphi$ gives an isomorphism in $\Tor^{\G}_X$. 
\end{proof}

\begin{cor}\label{betaeq}
The functor $\beta : \Hc^1(X; \G) \too  \Tor^{\G}_X$ is an  equivalence of categories. \qed
\end{cor}

\end{document}